\documentclass[10pt]{amsart}
\usepackage{amssymb, booktabs, multirow, dashrule}

\title[Rigidity of certain PCF endomorphisms]{Rigidity and height bounds for certain post-critically finite endomorphisms of $\PP^N$}
\author{Patrick Ingram}
\date{\today}
\address{Colorado State University, Fort Collins}

\newcommand{\QQ}{\mathbb{Q}}
\newcommand{\ZZ}{\mathbb{Z}}
\newcommand{\CC}{\mathbb{C}}
\newcommand{\RR}{\mathbb{R}}
\newcommand{\FF}{\mathbb{F}}
\newcommand{\PP}{\mathbb{P}}
\renewcommand{\AA}{\mathbb{A}}

\newcommand{\Ocal}{\mathcal{O}}

\newcommand{\Spec}{\operatorname{Spec}}

\newcommand{\PGL}{\operatorname{PGL}}

\newcommand{\origin}{\mathbf{0}}
\newcommand{\Ind}{\operatorname{Ind}}
\newcommand{\Res}{\operatorname{Res}}

\newcommand{\Div}{\operatorname{Div}}
\newcommand{\Proj}{\operatorname{Proj}}

\renewcommand{\phi}{\varphi}
\renewcommand{\epsilon}{\varepsilon}

\newcommand{\moduli}[2]{\mathsf{M}^{#1}_{#2}}			% the moduli space of morphisms
\newcommand{\pow}[2]{\operatorname{Pow}^{#1}_{#2}}		% the parameter space of monic polys
\newcommand{\powm}[2]{{{\mathsf{P}}}^{#1}_{#2}}				% the moduli space of monic polys
\newcommand{\RPE}[2]{\mathsf{RPE}^{#1}_{#2}}				% the moduli space of polys
\newcommand{\homspace}[2]{\operatorname{Hom}^{#1}_{#2}}	% the parameter space of morphisms
\newcommand{\ph}{\mathcal{B}}

\newtheorem{theorem}{Theorem}[section]
\newtheorem{lemma}[theorem]{Lemma}
\newtheorem{conjecture}[theorem]{Conjecture}
\newtheorem{corollary}[theorem]{Corollary}

\theoremstyle{remark}

\theoremstyle{definition}

\begin{document}
\begin{abstract}
The holomorphic map $f:\PP^N\to\PP^N$ is called post-critically finite (PCF) if the forward image of the critical locus, under iteration of $f$, has algebraic support. In the case $N=1$, a deep result of Thurston implies that there are no algebraic families of PCF morphisms, other than a well-understood exceptional class known as the flexible Latt\`{e}s maps. This note proves a corresponding result in arbitrary dimension, for a certain subclass of morphism. Specifically, we restrict attention to morphisms  $f:\PP^N\to\PP^N$ of degree $d\geq 2$, with a totally invariant hyperplane $H\subseteq\PP^N$, such that the restriction of $f$ to $H$ is the $d$th power map in some coordinates. This defines a subvariety $\powm{N}{d}\subseteq\moduli{N}{d}$
 of the space of coordinate-free endomorphisms of $\PP^N$. We prove that there are no families of PCF maps in $\powm{N}{d}$, and derive several related arithmetic results.
\end{abstract}
\maketitle

\section{Introduction}

A fundamental maxim of complex holomorphic dynamics indicates that one understands dynamical systems largely by understanding their critical orbits, that is, the forward orbits of their ramification loci.
Given this, maps for whom all of these orbits are finite take on a special importance.
Let $N\geq 1$, and let $f:\PP^N\to\PP^N$ be a morphism of degree $d\geq 2$.  We will say that $f$ is \emph{post-critically finite} (PCF) if and only if the forward orbit of the ramification locus of $f$, under the action of $f$, is supported on a finite union of algebraic hypersurfaces.
In the one-dimensional case, flexible Latt\`{e}s maps present an important class of PCF morphisms.
A fundamental result of Thurston~\cite{thurston} shows that, other than these Latt\`{e}s examples, univariate PCF maps do not come in families, but rather constitute a countable union of 0-dimensional subvarieties of the appropriate moduli space.

PCF morphisms $f:\PP^N\to\PP^N$ have been studied when $N\geq 2$ \cite{fs, jonsson, koch, rong, ueda}, but so far little is known about families of such maps. We will call a morphism $f:\PP^N\to\PP^N$ a \emph{monic polynomial} if there is a hyperplane $H\subseteq\PP^N$ which is totally invariant under $f$, and such that the restriction of $f$ to $H$ is the $d$th-power map, relative to some coordinates. For each $N\geq 1$ and $d\geq 2$ this defines a subvariety $\powm{N}{d}\subseteq \moduli{N}{d}$ of the moduli space of endomorphisms of $\PP^N$ of degree $d$ ($\powm{1}{d}$ is the usual space of univariate polynomials).
Our first result is a Thurston-type rigidity statement for the PCF points in this subvariety of the moduli space.
\begin{theorem}\label{th:rigidity}
Let $k$ be an algebraically closed field of characteristic 0 or $p>d$. Then the locus of post-critically finite maps in $\powm{N}{d}(k)$ is a countable union of $0$-dimension subvarieties.
\end{theorem}

The family of maps to which this theorem applies admits an elementary description. If $f\in\powm{N}{d}$ then $f$ has a model $f=[f_0:\cdots :f_N]$ in some homogeneous variables $[x_0:\cdots :x_N]$ such that $f_N=x_N^d$, and such that for each $0\leq i<N$ we have $f_i=x_i^d+x_Ng_i(x_0, ..., x_N)$ for some homogeneous form $g_i$ of degree $d-1$. We note that the hypothesis on the characteristic of $k$ in Theorem~\ref{th:rigidity} is clearly necessary; for example if $p<d$, then
\[f=\left[x_0^d+a_1x_0^px_N^{d-p}:\cdots x_{N-1}+a_Nx_{N-1}^px_N^{d-p}:x_N^d\right]\]
fixes its critical locus for any $a_1, ..., a_N\in k$, giving a map $\AA^N(k)\to\powm{N}{d}(k)$ which lands entirely in the PCF locus, and turns out to have finite fibres.

Theorem~\ref{th:rigidity} is arguably of most interest over the field $k=\CC$, but the proof is largely algebraic, and algebraically closed fields of positive characteristic introduce no particular challenges (given the restrictions in the statement of the theorem). It is not hard to show, from Theorem~\ref{th:rigidity}, that the PCF points in $\powm{N}{d}(\CC)$ are in fact contained in $\powm{N}{d}(\overline{\QQ})$. In particular, every PCF map $f\in\powm{N}{d}(\CC)$ is defined over some finite extension of the prime field $\QQ$. This brings the problem of studying monic PCF polynomial endomorphisms of $\PP^N_\CC$ into the realm of arithmetic geometry.

Our next result is arithmetic in nature, and to state it we define two pieces of notation. If $\pow{N}{d}$ is the parameter space of monic polynomials (a finite cover of $\powm{N}{d}$ defined below), then there is a Weil height corresponding to a particular weighted projective completion of $\pow{N}{d}$ which we will denote $h_{\mathrm{Weil}}$. In addition, we define below a non-negative real-valued function $h_\mathrm{crit}:\pow{N}{d}\to\RR$ corresponding to the arithmetic ``escape rate'' of the ramification locus; the definition is based on local analysis, and satisfies the property that $h_\mathrm{crit}(f)=0$ when $f$ is PCF (that is, finite orbits do not escape by this measure).  Although $h_\mathrm{crit}$ is defined in an \emph{ad hoc}, non-geometric way, and hence has no apparent reason to relate to the geometric Weil height, it turns out that the two functions are essentially the same.

\begin{theorem}\label{th:heights}
For $f\in \pow{N}{d}(\overline{\QQ})$ we have
\[h_{\mathrm{crit}}(f)= h_{\mathrm{Weil}}(f)+O(1).\]
In particular, the locus of PCF maps in $\pow{N}{d}(\overline{\QQ})$ is a set of bounded height.
\end{theorem}

The arithmetic properties of PCF maps have been studied previously, although only in dimension $N=1$. Specifically, the author~\cite{pcfpoly} established a weaker version of Theorem~\ref{th:heights} for univariate polynomials, while Epstein~\cite{epstein}, Levy~\cite{alon:thurston}, and Silverman~\cite{js:thurston} have given various algebraic proofs of special cases of Theorem~\ref{th:rigidity} when $N=1$. In the realm of rational functions, Benedetto, the author, Jones, and Levy~\cite{pcfrat} have shown that univariate PCF maps reside in a set of bounded height, although a height relation as strong as that in Theorem~\ref{th:heights} seems out of reach. Meanwhile, the equidistribution results of Favre and Gauthier~\cite{favre} and the unlikely intersection results of Baker and DeMarco~\cite{bd, bd2} have made use of this sort of arithmetic information. It is hoped that Theorem~\ref{th:heights} will provide a first step toward generalizing work in these areas to several variables.

As noted, it follows quickly from Theorem~\ref{th:rigidity} that every PCF map $f\in\powm{N}{d}$ is defined over some finite extension of $\QQ$. The next result, an immediate corollary of Theorem~\ref{th:heights} in light of the Northcott property for $h_\mathrm{Weil}$, shows that the degrees of these extensions offer a finite stratification of the PCF maps in this class.

\begin{corollary}\label{cor:finiteness}
Let $N, D\geq 1$ and let $d\geq 2$. Then there are only finitely many  monic PCF polynomials $f:\PP^N\to\PP^N$ of degree $d$, defined over number fields of degree at most $D$, up to change of coordinates.
\end{corollary}

The function $h_\mathrm{crit}$ in Theorem~\ref{th:heights} is constructed by first associating to each $f\in\pow{N}{d}$ a canonical height $\hat{h}_f$ defined on certain divisors of $\PP^N$, and then letting $h_\mathrm{crit}(f)=\hat{h}_f(C_f)$, where $C_f$ is the critical locus of $f$.  We note that heights of divisors and subvarieties have been considered before, most notably by Bost, Gillet, and Soul\'{e}~\cite{bgs}, by Philippon~\cite{philippon}, by Moriwaki~\cite{moriwaki}, and by Zhang~\cite{zhang} but these heights do not seem to lend themselves to the questions at hand. Our construction is less general, but tailored to the present setting, and sufficiently explicit to be used in computation; it would be of some interest to relate the height constructed here to previously constructed ones. We confirm in Theorem~\ref{th:canheightprops} below that the canonical height function $D\mapsto \hat{h}_f(D)$ has most of the properties that one would expect.

The results in this paper are entirely effective, and although it is not clear whether or not they allow one to effectively decide whether a given morphism is PCF, in practice they often do. In order to illustrate this, we give an explicit version of the above corollary in a special case.
\begin{theorem}\label{th:computations}
For any $a, b, c, d\in \QQ$, let $f_{a, b, c, d}:\PP^2_\QQ\to\PP^2_\QQ$ be the morphism extending
\[f_{a, b, c, d}(x, y)=(x^2+ax+by, y^2+cx+dy).\]
Then $f_{a, b, c, d}$ is PCF if and only if 
\begin{multline*}
(a, b, c, d)\in \Big\{
(0, 0, 0, 0), (0, 0, 0, -2), (-2, 0, 0, -2),\\  (0, 0, -1, 0), (0, 0, -2, 0),(0, -2, -2, 0)
\Big\}
\end{multline*}
or $f_{a, b, c, d}$ is related to one of these examples by a change of coordinates.
\end{theorem}

Note that the first three examples in Theorem~\ref{th:computations} are split, in the sense of consisting of two univariate polynomials acting independently, while the fourth and fifth are skew products in the sense of \cite{skewj}. The last example is a generalized Chebyshev map  \cite{uchimura, veselov}. See Table~\ref{table:examples} in Section~\ref{sec:comp} for more details on these examples.

The height bound in Theorem~\ref{th:heights} is established by decomposing both heights into local contributions, and establishing appropriate bounds at each place. These bounds work out in such a way as to offer interesting results for the local dynamics at each place.

Over $\CC$, our methods amount to defining an escape-rate function $G_f$ on a certain class of divisors on $\PP^N$, and considering the value $G_f(C_f)$, that is, the escape rate of the critical locus. Note that in the case $N=1$ and $d=2$, the relation $G_f(C_f)=0$ defines the Mandelbrot set in $\powm{1}{2}(\CC)= \CC$, and so our next theorem can be seen as a generalization of the compactness of this set.
\begin{theorem}\label{th:compactness}
For any $B\in\RR$, the subset of $\pow{N}{d}(\CC)$ on which $G_f(C_f)\leq B$ is compact.
\end{theorem}

In the case of non-archimedean local fields, recall that a morphism $f:\PP^N\to\PP^N$ defined over a local field $K$ has \emph{good reduction} if it extends to a scheme morphism $f:\PP^N_\Ocal\to\PP^N_\Ocal$ over $\Spec(\Ocal)$, where $\Ocal\subseteq K$ is the ring of integral elements. We will say that $f$ has \emph{potentially good reduction} if some change of coordinates over $\overline{K}$ yields good reduction.
\begin{theorem}\label{th:goodred}
Let $K$ be a local field with residue characteristic $0$ or $p>d$. If $f:\PP^N\to\PP^N$ is a monic PCF polynomial of degree $d$ over $K$, then $f$ has potentially good reduction.
\end{theorem}
If $f:\PP^N_K\to\PP^N_K$ and $f(P)=P$, then $f$ induces an action on the tangent space $d_Pf:\Theta_{\PP^N, P}\to\Theta_{\PP^N, P}$. We say that the point $P$ is \emph{strongly non-repelling} if every eigenvalue $\lambda$ of this action satisfies $|\lambda|\leq1$. An $n$-periodic point is strongly non-repelling if it is a strongly non-repelling fixed point of $f^n$. The following is related to Theorem~7.1 of \cite{pcfrat} in the case $N=1$.
\begin{corollary}\label{cor:repelling}
Let $K$ be a local field with residue characteristic $0$ or $p>d$. If $f:\PP^N\to\PP^N$ is a monic PCF polynomial of degree $d$ over $K$, then all periodic points of $f$ are strongly non-repelling.
\end{corollary}
A version of this corollary can also be established over $\CC$, but the statement is much weaker, specifically that the size of the eigenvalues of the action on the tangent space at points of period $n$ for PCF maps are bounded above by a constant depending just on $d$, $N$, and $n$. In other words, the periodic points of PCF monic polynomials are ``not too strongly repelling,'' in a uniform sense.

We note that our focus on monic polynomial endomorphisms falls into a larger perspective which has been studied before, for instance by Bedford and Jonsson~\cite{bj}. In general, a \emph{regular polynomial endomorphism} $f:\PP^N\to\PP^N$ is a morphism with a totally invariant hyperplane $H$, and we denote by $\RPE{N}{d}$ the space of such maps of degree $d\geq 2$ with a marked invariant hyperplane (a subvariety of $\moduli{N}{d}$ but for the extra marked structure). Since $H\cong \PP^{N-1}$, there is a natural projection $\pi:\RPE{N}{d}\to\moduli{N-1}{d}$, given by restriction to the hyperplane. It is easy to see that PCF maps are sent to PCF maps, and that the fibre above any PCF map contains at least one PCF map. So, one approach to studying the PCF maps in $\RPE{N}{d}$ is to study the fibres over PCF points in $\moduli{N-1}{d}$. Theorem~\ref{th:rigidity} states that there are no families of PCF maps in one particular fibre of this fibration, namely the fibre above the $d$th power map on $\PP^{N-1}$. Although this fibre is certainly a very special one, it seems reasonable to speculate that there are no fibral families of PCF maps in general, in other words, that PCF maps in $\RPE{N}{d}$ admit a sort of relative rigidty over the base $\moduli{N-1}{d}$.
\begin{conjecture}\label{conj:relativerigidity}
Let $X/\CC$ be an affine curve, let $\sigma:X\to\RPE{N}{d}$ land entirely in the PCF locus, and suppose that $\pi\circ\sigma:X\to\moduli{N-1}{d}$ is constant. Then $\sigma$ is constant.
\end{conjecture}
 If this conjecture were true, it would then follow from Thurston's rigidity theorem for $\moduli{1}{d}$ that the only families of regular polynomial endomorphisms $f:\PP^2\to\PP^2$ are those that are ``Latt\`{e}s at infinity.''
Some of the arguments in this paper seem to use properties distinct to the fibre above the power map, but others do not. One key ingredient, for instance, is a regularity of Green's functions as $f\in \powm{N}{d}$ varies, and one of the the main ingredients in \cite{pi:specN} is the observation that a similar regularity holds on any given fibre of this projection.

%%%%%%%%%%%%%%%%%%%%%%%%%%%%%%
%%%%%%%%%%%%%%%%%%%%%%%%%%%%%%
%%%%%%%%%%%%%%%%%%%%%%%%%%%%%%
%%%%%%%%%%%%%%%%%%%%%%%%%%%%%%
%%%%%%%%%%%%%%%%%%%%%%%%%%%%%%
%%%%%%%%%%%%%%%%%%%%%%%%%%%%%%
%%%%%%%%%%%%%%%%%%%%%%%%%%%%%%
%%%%%%%%%%%%%%%%%%%%%%%%%%%%%%

This paper is organized as follows.
In Section~\ref{sec:spaces} we define the spaces $\pow{N}{d}$ and $\powm{N}{d}$, as well as some associated parameter spaces, more carefully, and study their geometric properties. In order to keep considerations mostly independent of characteristic, we work in the category of schemes over $\ZZ[\frac{1}{2}, ..., \frac{1}{d}]$, rather than varieties over $\CC$.
Section~\ref{sec:local} is devoted to the construction of a Green's-like function on divisors, relative to a given monic polynomial endomorphism, and from its properties we prove Theorems~\ref{th:rigidity} and Theorem~\ref{th:goodred}.  Here we work entirely over algebraically closed fields complete with respect to some non-archimedean absolute value, while in Section~\ref{sec:complex} we work out similar (but messier) results over $\CC$ and prove Theorem~\ref{th:compactness}.
In Section~\ref{sec:global} we turn our attention to global fields, and prove Theorem~\ref{th:heights}.
Since the results in this paper are more-or-less computationally effective, Section~\ref{sec:comp} gives explicit computational consideration to the PCF monic quadratic polynomials $f:\PP^2\to\PP^2$ defined over $\QQ$.

%%%%%%%%%%%%%%%%%%%%%%%%%%%%%%
%%%%%%%%%%%%%%%%%%%%%%%%%%%%%%
%%%%%%%%%%%%%%%%%%%%%%%%%%%%%%
%%%%%%%%%%%%%%%%%%%%%%%%%%%%%%
%%%%%%%%%%%%%%%%%%%%%%%%%%%%%%
%%%%%%%%%%%%%%%%%%%%%%%%%%%%%%
%%%%%%%%%%%%%%%%%%%%%%%%%%%%%%
%%%%%%%%%%%%%%%%%%%%%%%%%%%%%%

\section{The space $\pow{N}{d}$}\label{sec:spaces}

Fixing $d\geq 2$, all geometric objects in the section can be considered at the level of schemes over $\Spec(R)$, where $R=\ZZ[\frac{1}{2}, ..., \frac{1}{d}]$, although the reader loses little intuition in conceiving of them as varieties over $\CC$.
We recall some notation of (see~\cite{barbados}). Morphisms $\PP^N\to\PP^N$ of degree $d\geq 2$ are naturally parametrized by their coefficients, and we denote this parameter space by $\homspace{N}{d}\subseteq \PP^{\binom{N+d}{d}(N+1)-1}$. There is a natural action of $\PGL_{N+1}$ on this space, corresponding to conjugation by a change of variables on $\PP^N$, and the quotient of $\homspace{N}{d}$ by this action is denoted by $\moduli{N}{d}$.

By a \emph{multi-index} $I$ of dimension $N$ we simply mean an $(N+1)$-tuple $I=(I_0, ..., I_N)$ of non-negative integers, and we set $|I|=I_0+\cdots +I_N$. If $\mathbf{x}=(x_0, ..., x_N)$ is a tuple of variables, then we will write $\mathbf{x}^I$ for the monomial $x_0^{I_0}\cdots x_N^{I_N}$. We will denote by $\Ind(N, d)$ the set of all $N$-dimensional multi-indices $I$ satisfying $|I|=d$, and write $\Ind^*(N, d)$ for those indices with $0< I_N<d$. 
 Note that, by standard combinatorial identities,
\[\#\Ind^*(N, d)=\#\Ind(N, d)-\#\Ind(N-1, d)-1=\binom{N+d}{d}-\binom{N-1+d}{d} - 1.\]
For each $0\leq i<N$ and each $I\in \Ind^*(N, d)$ we will introduce an indeterminate $a_{i, I}$, and set
\[A=R[a_{i, I}:0\leq i< N, I\in\Ind^*(N, d)].\]
We will write $\pow{N}{d}=\Spec(A)$, as a scheme over $R$. We will also view $A$ as a graded ring, with a grading defined by letting $a_{i, I}$ have weight $I_N$, and we will consider below the (weighted) projective space $\Proj(A)$ with respect to this gradation.

For any point $P=(a_{i, I})\in \pow{N}{d}$, we define a morphism $f_P:\PP^N\to\PP^N$ by
\begin{equation}\label{eq:genericpoly}
f_P(\mathbf{x})=\left[x_0^d+\sum_{I\in\Ind^*(N, d)}a_{0, I}\mathbf{x}^I:\cdots :x_{N-1}^d+\sum_{I\in\Ind^*(N, d)}a_{N-1, I}\mathbf{x}^I:x_N^d\right],
\end{equation}
giving an embedding $\pow{N}{d}\to\homspace{N}{d}$ over $\Spec(R)$. Note that, over an algebraically closed field, every PCF monic polynomial (as defined in the introduction) is equivalent to a map of the above form, up to a change of coordinates. In other words, if we let $\powm{N}{d}\subseteq \moduli{N}{d}$ denote the image of $\pow{N}{d}$ under the quotient by $\PGL_{N+1}$, then $\powm{N}{d}$ corresponds exactly to the set of PCF monic polynomials.
\begin{lemma}
The map $\pow{N}{d}\to\powm{N}{d}$ is finite.
\end{lemma}

\begin{proof}
The fibres of the map $\pow{N}{d}\to\powm{N}{d}$ consist of collections of polynomials $f\in\pow{N}{d}$ that are $\PGL_{N+1}$-conjugate, so suppose that $f, g\in\pow{N}{d}$ and that $\phi\in\PGL_{N+1}$ satisfies $\phi f = g\phi$. If $\phi$ fixes $H$, then $\phi$ has block form \[\phi=\begin{pmatrix}\psi & b \\ 0 & 1 \end{pmatrix},\] where $\psi$ is an automorphism of the power map on $H$, $b$ is an affine fixed point of $g$. The automorphisms of the power map are generated by the permutation matrices and diagonal matrices whose eigenvalues are all $(d-1)$th roots of unity, and $g$ has only finitely many fixed points, so there are only finitely many possibilities for $f$.

Now, if $\phi$ does not fix $H$, then $\phi H$ is another invariant hyperplane for $g$. But $g$ has at most $N+1$ invariant hyperplanes, and so $\phi$ is in one of at most $N+1$ conjugacy classes of the finite set of matrices mentioned above.
\end{proof}

We will now consider the action of $f:\PP^N\to\PP^N$ on divisors on $\PP^N$. We let $H$ denote the invariant hyperplane of $f$, which we have taken to be $H:\{x_N=0\}$. We let $\Div(\PP^N)$ denote the usual group of Weil divisors on $\PP^N$, and $\Div^+(\PP^N)$ the semigroup of effective divisors. For divisors defined over a given field $k$ over $R$, we use $\Div_k(\PP^N)$ and $\Div^+_k(\PP^N)$, respectively. 
We define a sub-semigroup $\Div^*(\PP^N)\subseteq\Div^+(\PP^N)$ by declaring that $D\in\Div^*(\PP^N)$ if and only if  $D$ intersects $H$ only where the other coordinate hyperplanes intersect $H$. Equivalently, $D\in\Div^*(\PP^N)$ if and only if $D$ is defined by the vanishing of a homogeneous form $F_D(x_0, ..., x_N)$  satisfying 
\[F_D(x_0, ..., x_{N-1}, 0)=\alpha\prod_{i=0}^{N-1}x_i^{e_i},\]
for some $e_i\geq 0$ and some unit $\alpha$, a condition which is clearly independent of the defining equation. One easily checks that $\Div^*(\PP^N)$ is closed under pulling-back by $f$.

In order to describe the operation of pushing forward divisors, we recall the Macaulay resultant. 
The following theorem is a description of this construction in the requisite generality, as developed by Jouanolou~\cite{jouanolou}.
To state the theorem, let $k$ be a commutative ring with identity, let $d_0, ..., d_n\geq 1$ be fixed, let
\[k'=k[u_{i, I}: 0\leq i\leq n, |I|=d_i].\]
We define a polynomial over $k'$ in the variables $x_0, ..., x_n$ by
\[F_i(\mathbf{x})=\sum_{|I|=d_i} u_{i, I}\mathbf{x}^I,\]
and let
$X=\operatorname{Proj}\left(k'[x_0, ..., x_n]/(F_0, ..., F_n)\right).$ Then the canonical map $k'\to \Gamma(X, \Ocal_X)$ has a kernel $\mathfrak{a}\subseteq k'$.

\begin{theorem}[The Macaulay Resultant \cite{jouanolou}]
The ideal $\mathfrak{a}$ is principal, and admits a generator
\[\Res_{x_0, ..., x_n}(F_0, ..., F_n)\in k'\]
which is unique given the stipulation that $\Res_{x_0, ..., x_n}(x_0^{d_0}, ..., x_n^{d_n})=1$.
If $D=d_1d_2\cdots d_n$, then the polynomial $\Res_{x_0, ..., x_n}(F_0, ..., F_n)$ over $k'$ is homogeneous of degree $D/d_j$ in the variables $u_{j, I}$, and has total degree 
$D(1/d_1+\cdots+1/d_n)$. Furthermore, the function $\Res_{x_0, ... x_N}$ is multiplicative in each variable.
\end{theorem}

We now describe the operation of pushing forward a divisor via the polynomial endomorphism $f:\PP^N\to\PP^N$. For any ring $k$ over $R$ and any homogeneous form $F(x_0, ..., x_N)$, we define a homogeneous form $\Res(F, f)\in k[y_0, ..., y_N]$ by
\begin{multline*}
\Res(F, f)(y_0, ..., y_{N-1}, 1)\\=\Res_{x_0, ..., x_N}(F, y_0x_N^d-f_0(x_0, ..., x_{N}), ..., y_{N-1}x_N^d-f_{N-1}(x_0, ..., x_N)).\end{multline*}
Since it is at times more convenient to work with homogeneous coordinates, we note that the homogeneity of the Macaulay resultant ensures that
\begin{multline*}
\Res(F, f)(y_0, ..., y_{N-1}, y_N)y_N^{N\deg(F)d^{N-1}}=\\\Res_{x_0, ..., x_N}(F, y_0x_N^d-y_Nf_0(x_0, ..., x_{N}), ..., y_{N-1}x_N^d-y_Nf_{N-1}(x_0, ..., x_N)).\end{multline*}

For any effective divisor $D$ defined by $F_D=0$, we may now define $f_*(D)$ to be the divisor defined by $\Res(F_D, f)=0$. Note that the multiplicativity of the Macaulay resultant ensures that $f_*:\Div^+(\PP^N)\to\Div^+(\PP^N)$ is $\ZZ$-linear.
We will write $f(D)$ for the radical of $f_*(D)$, where the \emph{radical} of $e_1D_1+\cdots +e_rD_r$ is $D_1+\cdots +D_r$ whenever the $D_i$ are distinct, and the $e_i$ positive. It is also easy to check that the restriction to $H$ of $f_*(D)$ is the push-forward of the restriction of $D$ by the restriction of $f$, and so $f_*:\Div^*(\PP^N)\to\Div^*(\PP^N)$. Since an effective divisor is in $\Div^*(\PP^N)$ if and only if every one of its summands is, it is also clear that $\Div^*(\PP^N)$ is closed under taking radicals.

We note that, by the definition of the MacCaulay resultant, over an algebraically closed field $k$ we have $P\in f_*(D)$ if and only if there is a point $Q\in D$ with $f(Q)=P$. In other words, the set of $k$-rational points on $f_*(D)$ (equivalently $f(D)$) is precisely the image under $f$ of the set of $k$-rational points on $D$.

We can now precisely define what it means for a divisor to be preperiodic under $f$.
We say that the divisor $D\in\Div^+(\PP^N)$ is \emph{preperiodic} for the morphism $f:\PP^N\to\PP^N$ if and only if the sequence $f^n(D)$ takes only finitely many values as $n\to\infty$ or, equivalently, if there are only finitely many irreducible divisors which occur as summands of $f^n_*(D)$ as $n\to\infty$. This is equivalent to the usual definition over $\CC$, which defines $D$ to be preperiodic if and only if the set
\[D \cup f(D)\cup f^2(D)\cup \cdots\] is an algebraic variety.

To any $f\in\pow{N}{d}$, we associate a homogeneous \emph{Jacobian form}
\[J_f(x_0, ..., x_N)=\det\begin{pmatrix}
\frac{\partial f_0}{\partial x_0} & \frac{\partial f_0}{\partial x_1} & \cdots & \frac{\partial f_0}{\partial x_{N-1}}\\
\frac{\partial f_1}{\partial x_0} & \frac{\partial f_1}{\partial x_1} & \cdots & \frac{\partial f_1}{\partial x_{N-1}}\\
\vdots & \vdots && \vdots\\
\frac{\partial f_{N-1}}{\partial x_0} & \frac{\partial f_{N-1}}{\partial x_1} & \cdots & \frac{\partial f_{N-1}}{\partial x_{N-1}}\\
\end{pmatrix}\in A[x_0, ..., x_N],\]
which one can check is a form of degree $N(d-1)$ in the variables $x_0, ..., x_N$. The \emph{critical divisor} of $f$ is the divisor $C_f$ defined by $\{J_f=0\}$, and we observe that $C_f+(d-1)H$ is the ramification divisor of the map $f:\PP^N\to\PP^N$ (our restriction of the characteristic avoids issues of wild ramification).  Note that since each partial derivative has the property that the coefficient of $\mathbf{x}^I$ is homogeneous of weight $I_N$ (with respect to the grading defined on $A$) the same is true of the determinant $J_f$.

\begin{lemma}
If $J_f$ is the Jacobian form of $f$, defining $C_f$, then write the form $\Res(J_f, f)$, which defines $f_*(C_f)$ as
\begin{equation}\label{eq:branchdefinition}
\Res(J_f, f)(\mathbf{y})=\prod_{i=0}^{N-1}y_i^{d^{N-1}(d-1)}+\sum_{J_N\neq 0}b_{J}\mathbf{y}^J.\end{equation}
This form has degree $d^{N-1}(d-1)N$ in $\mathbf{y}$, and for each multi-index $J$ the coefficient $b_J\in A$ is either $0$ or homogeneous of weight $dJ_N$.
\end{lemma}

\begin{proof}
The leading term, and the homogeneity and degree in $\mathbf{y}$, follow immediately from the properties of the Macaulay resultant. It remains to show that $b_J$ has degree $dJ_N$ with respect to the grading on $A$.

We first note two basic properties of the Macaulay resultant, both following from the standard homogeneity properties, necessary for the proof, namely
\[
\Res_{x_0, ..., x_n}(F_0,..., \beta F_i, ..., F_n)\\=\beta^{D/d_i}\Res_{x_0, ..., x_n}(F_0,..., F_n)
\]
and
\begin{multline*}
\Res_{x_0, ..., x_n}(F_0(x_0, ..., \beta x_i, ..., x_n),...,  F_n(x_0, ..., \beta x_i, ..., x_n))\\=\beta^{D}\Res_{x_0, ..., x_n}(F_0, ..., F_n),
\end{multline*}
where $d_i=\deg(F_i)$ and $D=d_1d_2\cdots d_n$.

If we define an action of $\mathbb{G}_\mathrm{m}$ on $A$ by $\alpha\cdot a_{i, I}=\alpha^{I_N} a_{i, I}$, then  the non-zero $g\in A$ satisfying $\alpha\cdot g=\alpha^w g$ are precisely the homogeneous elements of degree $w$. Extending this action to polynomials over $A$ by letting $\mathbb{G}_\mathrm{m}$ act trivially on the variables,  what we wish to show is that if $G=\Res(J_f, f)$, normalized as above, then
\begin{eqnarray*}
\alpha\cdot G(y_0, ..., y_N)&:=& \prod_{i=0}^{N-1}y_i^{d^{N-1}(d-1)}+\sum_{I}\left(\alpha\cdot b_{I}\right)\mathbf{y}^I\\
&=&\prod_{i=0}^{N-1}y_i^{d^{N-1}(d-1)}+\sum_{I}\left(\alpha ^{dI_N}b_{I}\right)\mathbf{y}^I\\
&=& G(y_0, ..., \alpha^d y_N).
\end{eqnarray*}
Since the resultant is a polynomial in the coefficients of the inputs, computing the resultant commutes with the above-defined action by $\mathbb{G}_\mathrm{m}$, and we have
\begin{multline*}
\alpha\cdot G(y_0, ..., y_N)=y_N^{-N^2(d-1)d^N}\Res_{x_0, ..., x_N}(\alpha\cdot J_f, y_0x_N^d-y_N(\alpha\cdot f_0),\\  ..., y_{N-1}x_N^d-y_N(\alpha\cdot f_{N-1})).
\end{multline*}
Now, since the coefficient of $\mathbf{x}^I$ in $f_i$ is homogeneous of degree $I_N$, we see that \[\alpha\cdot f_i(x_0, ..., x_N)=f_i(x_0, ..., x_{N-1}, \alpha x_N),\] and similarly for $J_f$.
Thus, by the above homogeneity properties, we have (for $D=(d-1)Nd^{N}$ the product of degrees of $J_f$ and the $y_ix_N^d-y_Nf_i$ as forms in $\mathbf{x}$)
\begin{eqnarray*}
\alpha\cdot G(y_0, ..., y_N)&=&y_N^{-N^2(d-1)d^{N-1}}\Res\Big(J_f(x_0, ..., x_{N-1}, \alpha x_N),\\ 
&&\qquad y_0x_N^d-y_Nf_0(x_0, ..., x_{N-1}, \alpha x_N), \\ &&\qquad..., y_{N-1}x_N^d-y_Nf_{N-1}(x_0, ..., x_{N-1}, \alpha x_N)\Big)\\
&=& y_N^{-N^2(d-1)d^{N-1}}(\alpha^d)^{-N^2(d-1)d^{N-1}}\Res\Big(J_f(x_0, ..., x_{N-1}, \alpha x_N),\\ 
&&\qquad \alpha^dy_0x_N^d-\alpha^dy_Nf_0(x_0, ..., x_{N-1}, \alpha x_N), \\ &&\qquad..., \alpha^dy_{N-1}x_N^d-\alpha^dy_Nf_{N-1}(x_0, ..., x_{N-1}, \alpha x_N)\Big)\\
&=&(\alpha^d y_N)^{-N^2(d-1)d^N}\Res\Big(J_f(x_0, ..., x_{N-1}, x_N),\\ 
&&\qquad y_0x_N^d-\alpha^dy_Nf_0(x_0, ..., x_{N-1},  x_N), \\ &&\qquad..., y_{N-1}x_N^d-\alpha^dy_Nf_{N-1}(x_0, ..., x_{N-1},  x_N)\Big)\\
&=&G(y_0, ..., y_{N-1}, \alpha^d y_N).
\end{eqnarray*}
\end{proof}

Now for each $J\in \Ind(N, d^{n-1}(d-1)N)$ with $J_N\neq 0$ introduce an indeterminate $b_J$, and let \[B=R\Big[b_J:J\in \Ind(N, d^{N-1}(d-1)N), J_N\neq 0\Big]\] viewed as a graded $R$-algebra with $b_J$ homogeneous of degree $J_N$. Viewing $b_J$ as a coefficient of the above form gives a homomorphism $B\to A$ taking elements of degree $w$ to elements of degree $dw$. In other words, the construction $f\mapsto f_*(C_f)$ gives rise to a rational map $\Proj A\to \Proj B$ of projective $R$-schemes, with degree $d$. In some sense, the crux of our argument is the observation that this map is regular.

\begin{lemma}\label{lem:crux}
The rational map $\Proj A\to\Proj B$ corresponding to $f\mapsto f_*(C_f)$ is a morphism.
\end{lemma}

\begin{proof}
If the claim is false, then there exists an algebraically closed field $k$ over $R$ and an $f\in\pow{N}{d}(k)$ such that $b_J(f)=0$ for all $J$, but $a_{i, I}(f)\neq 0$ for at least one pair $i, I$. In other words, if $H_i=\{x_i=0\}\subseteq \PP^N_k$, we have
 \[f_*(C_f)=d^{N-1}(d-1)(H_0+\cdots +H_{N-1})\] but $f(x_0,\cdots, x_N)\neq [x_0^d:\cdots: x_N^d]$.  We will show that this is impossible.

Given $f$ with the above properties, partition the support of $C_f$ as
\[\operatorname{Supp}(C_f)=S_1\cup\cdots\cup S_{N-1},\]
where for each $i$ and each $D\in S_i$, we have $f(D)=H_i$. Note that none of the $S_i$ can be empty, and that they are disjoint. Now, let $e_D\geq 0$ be defined for each $D\in\Div^*(\PP^N)$ such that $f^*(H_i)=\sum e_DD$. Note that, since our assumption of characteristic rules out wild ramification, it must be the case that $D$ occurs in $C_f$ with multiplicity $e_D-1$.  We have
\begin{eqnarray*}
dN&=&\sum_{i=0}^{N-1}\deg f^*(H_i)\\
&= &\sum_{i=0}^{N-1}\sum_{D\in\operatorname{Supp} (f^*(H_i))}e_D \deg(D) \\
&\geq &\sum_{i=0}^{N-1}\sum_{D\in S_i}e_D \deg(D)\\
&=&\sum_{i=0}^{N-1}\sum_{D\in S_i}(e_D-1) \deg(D)+\sum_{i=0}^{N-1}\sum_{D\in S_i}\deg(D)\\
&=&\deg(C_f)+\sum_{i=0}^{N-1}\sum_{D\in S_i}\deg(D).
\end{eqnarray*}
Since $\deg(C_f)=(d-1)N$, and since $\sum_{D\in S_i}\deg(D)\geq 1$ for each $i$, we see that each $S_i$ must consist of a single divisor of degree 1, say $S_i=\{D_i\}$. For each $i$, $D_i$ is a hyperplane containing both the hyperplane $H_i\cap H$ of $H$ and the point $\origin$, which leaves us only with the possibility $D_i=H_i$. Our function now has the property that $f^*(H_i)=dH_i$, for each $0\leq i \leq N$, in other words we have $f(x_0,\cdots, x_N)=[x_0^d:\cdots :x_N^d]$. This is a contradiction.
\end{proof}

Note that the morphism in Lemma~\ref{lem:crux} is not a natural object from the point of view of dynamics. The equivalence which collapses $\pow{N}{d}\setminus\{\origin\}$ to $\Proj A$ corresponds to pre-composition of polynomials by a scaling, without a corresponding post-composition, an operation which does not commute with iteration. In fact, our results are not based on studying the forward orbit of $C_f$ under $f$ so much as studying its immediate forward image. In light of this, the focus on $\Proj A$ is not particularly unusual.

\begin{lemma}\label{lem:nullstellensatz}
There exist an integer $e$ divisible by $(d-1)!$ and polynomials $g_{i, I, J}\in A$ homogeneous of degree $e-d!$ with respect to the gradation on $A$, such that for each
$0\leq i< N$ and multi-index $|I|=d$ with $I_N\neq 0, d$
 we have
\[a_{i, I}^{e/I_N}=\sum b_{J}^{(d-1)!/J_N}g_{i, I, J},\]
where the sum is over  multi-indices
$J$ with $|J|=d^{N-1}(d-1)N$.
\end{lemma}

\begin{proof}
This is a more-or-less standard application of Hilbert's Nullstellensatz, but we outline the details for the convenience of the reader. For a given $i$ and $I$, we dehomogenize all polynomials by setting $a_{i, I}=1$, and let $\mathfrak{a}\subseteq A$ be the ideal generated by the $b_J^{(d-1)!/J_N}$.

 If $1\not\in \mathfrak{a}$, then $\mathfrak{a}\subseteq \mathfrak{m}$ for some maximal ideal $\mathfrak{m}$ of $A$. Since the ring $R$ is Jacobson, the Nullstellensatz ensures that $\mathfrak{m}\cap R=pR$ for some $p>d$, and that $A/\mathfrak{m}$ is a finite extension of  $R/pR\cong \FF_p$.  But then in this finite field the polynomials $b_{J}$ have a common root with $a_{i, I}=1$, violating Lemma~\ref{lem:crux}. 
 
 It must be the case, then, that $1\in \mathfrak{a}$, and by rehomogenizing we may write 
\[a_{i, I}^{e/I_N}=\sum b_{J}^{(d-1)!/J_N}g_{i, I, J},\]
for some $e$ divisible by $I_N$, and some $g_{i, I, J}\in A$.  Since we are free to increase $e$, we may assume that it is the same value for every $a_{i, I}$.  Since $a_{i, I}^{e/I_N}$ is homogeneous of degree $e$, and $b_{J}^{(d-1)!/J_N}$ is homogeneous of degree $d!$, it must be the case that $g_{i, I, J}$ is homogeneous of degree $e-d!$.
\end{proof}

We close this section by confirming the first part of Theorem~\ref{th:rigidity}, namely that points $f\in \pow{N}{d}$ such that $C_f$ is preperiodic for $f$ consist of a countable union of subvarieties of the moduli space. The main content of this theorem, namely that these subvarieties are zero-dimensional, will be proven later.

\begin{lemma}
For every $m>n\geq 0$, the condition $\operatorname{Supp}(f^m(C_f))\subseteq \operatorname{Supp}(f^n(C_f))$ defines a non-empty closed subscheme $C_{n, m}\subseteq \pow{N}{d}$ over $R$. Similarly, the condition $\operatorname{Supp}(f^m(C_f))\subseteq \operatorname{Supp}(f^n(C_f))$ defines a non-empty closed subscheme $\mathsf{C}_{n, m}\subseteq \powm{N}{d}$
\end{lemma}

\begin{proof}
For any $N$ and $t\geq 1$, we let $V(N, t)\cong \PP^{\binom{N+t}{t}-1}$ denote the space of divisors of degree $t$ on $\PP^N$, parametrized by the coefficients of the defining forms. Note that, for any $s_1, ..., s_r$, the map
\[V(N, t_1)\times \cdots \times V(N, t_r)\to V\left(N, \sum s_it_i\right)\]
corresponding to the operation $(D_1, ..., D_r)\mapsto \sum s_i D_i$ is a morphism. Now fix $d_1, d_2$, and for any data $t_1, ..., t_r>0$, $s_{1, 1}, ..., s_{1, r}>0$, and $s_{2, 1}, ..., s_{2, r}\geq 0$ satisfying $\sum s_{i, j}t_j=d_i$, let
\[V(N, t_1)\times \cdots \times V(N, t_r)\to X_{\textbf{t}, \textbf{s}}\subseteq V(N, d_1)\times V(N, d_2)\]
denote the image of the above-described map. Then $Y_{d_1, d_2}\subseteq V(N, d_1)\times V(N, d_2)$, the union of $X_{\mathbf{t}, \mathbf{s}}$ over all appropriate data, is a Zariski closed subset, and a pair of divisors $(D_1, D_2)$ correspond to a point in $X$ if and only if $\deg(D_i)=d_i$ and $\operatorname{Supp}(D_2)\subseteq \operatorname{Supp}(D_1)$.

If $m>n\geq 0$, let $d_1=d^{n(N-1)}(d-1)$ and $d_2=d^{m(N-1)}(d-1)$. Using the Macaulay resultant, we have a map
\[\pow{N}{d}\to V\left(N, d_1\right)\times V\left(N, d_2\right)\]
defined by $f\mapsto (f^n_*(C_f), f^m_*(C_f))$ for each $n, m$. We let $C_{n, m}$ denote the inverse image of $Y_{d_1, d_2}$ under this map, so that $f\in C_{n, m}$ if and only if $\operatorname{Supp}(f^m(C_f))\subseteq \operatorname{Supp}(f^n(C_f))$. The subscheme $\mathsf{C}_{m, n}$ is the image of $C_{m, n}$ under the (finite) map $\pow{N}{d}\to\powm{N}{d}$.

To show that each of these is non-empty, let $F_c(z)=z^d+c$, and define
\[f[x_0:\cdots :x_N]=\left[x_N^dF_c(x_0/x_N):\cdots :x_N^dF_c(x_{N-1}/x_N):x_N^d\right].\]
A simple calculation shows that $f\in C_{m, n}$ if $F_c^m(0)=F_c^n(0)$, and it is well-known that there are solutions to this over $\CC$ for any $m>n\geq 0$.
\end{proof}

Note that if $N_1+\cdots N_r=N$, then there is a natural map \[\pow{N_1}{d}\times\cdots\times \pow{N_r}{d}\to\pow{N}{d}\] which we refer to as the direct product of monic polynomial maps (the same construction exists for regular polynomial endomoprhisms in general). We have shown that each $C_{n, m}$ is non-empty by exhibiting in it a direct product of univariate polynomials. This is somewhat unsatisfactory, and so we will point out that the PCF locus also contains
\[f(x_0, ..., x_N)=[x_0^d+\alpha_0x_{N-1}^{d-1}x_N:\cdots x_{N-2}+\alpha_{N-2}x_{N-1}^{d-1}x_N:x_{N-1}^d:x_N^d],\]
which is not a direct product of polynomial in fewer variables, whenever $z^d+\alpha_i$ is PCF for each $i$. 
One can verify this simply by noting that the critical locus of this map is the sum of the coordinate axes, while the direct image of 
$x_i^d-\gamma^d x_{N-1}^{d-1}x_N=0$ is $x_i^d-(\gamma^d+\alpha_i)^d x_{N-1}^{d-1}x_N=0$.
Though not a direct product, this polynomial is a skew product in the sense of \cite{skew, skewj}.

%%%%%%%%%%%%%%%%%%%%%%%%%%%%%%
%%%%%%%%%%%%%%%%%%%%%%%%%%%%%%
%%%%%%%%%%%%%%%%%%%%%%%%%%%%%%
%%%%%%%%%%%%%%%%%%%%%%%%%%%%%%
%%%%%%%%%%%%%%%%%%%%%%%%%%%%%%
%%%%%%%%%%%%%%%%%%%%%%%%%%%%%%
%%%%%%%%%%%%%%%%%%%%%%%%%%%%%%
%%%%%%%%%%%%%%%%%%%%%%%%%%%%%%

\section{Non-archimedean places}\label{sec:local}

We let $K$ be any algebraically closed field complete with respect to a non-trivial non-archimedean absolute value $|\cdot|$ associated to a valuation $v$. We will always assume that $\operatorname{char}(K)=0$ or $\operatorname{char}(K)>d$, where $d\geq 2$ is the fixed degree of the morphisms $f:\PP^N\to\PP^N$ under consideration, although the residual characteristic of $K$ is allowed to be arbitrary. In what follows, we will say that the valuation $v$ is \emph{$p$-adic} if and only if $|p|<1$. Given our assumptions on characteristic, this occurs just in case $K$ has characteristic $0$ and the valuation extends the $p$-adic valuation on $\QQ\subseteq K$, suitably normalized.

We begin by describing what will be the local contribution to our naive height on divisors.
If $D\in\Div^*(\PP^N)$ we let $F_D(x_0, ..., x_N)$ be the unique form defining $D$ which satsifies $F_D(x_0, ..., x_{N-1}, 0)=\prod_{i=0}^{N-1} x_i^{e_i}$ for some $e_i\geq 0$, and we set
\[\lambda_v(D)=\log^+\sup_{|\beta_0|=|\beta_1|=\cdots=|\beta_{N-1}|=1}\max\left\{|\beta_N|^{-1}:F(\beta_0, ..., \beta_N)=0\right\}.\]
For a given $f\in \pow{N}{d}(K)$ written as in \eqref{eq:genericpoly} we will also define
\[\ph_v(f)=\log^+\max\left\{|a_{i, I}|^{1/I_n}:0\leq i\leq N-1\text{ and } I\in \Ind^*(N, d)\right\}.\]

Our first lemma tells us, among other things, that $\lambda_v$ is indifferent to multiplicities, and hence to the distinction between $f(D)$ and $f_*(D)$.

\begin{lemma}\label{lem:lambdapropsnonarch}
\begin{enumerate}
\item\label{it:linear} For any $D_i\in\Div^*(\PP^N)$ and integers $e_i\geq 1$ we have \[\lambda_v\left(\sum e_iD_i\right)=\max\left\{\lambda_v (D_i)\right\}.\]
\item\label{it:lambdamu} For any $D\in\Div^*(\PP^N)$, if we write
\[F_D(x_0, ..., x_N)=\sum_{i=0}^d c_i(x_0, ..., x_{N-1})x_N^{d-i},\] and if $\|P\|$ denotes the Gauss norm of the polynomial $P$, then we have
\begin{equation}\label{eq:lambdaGauss}\lambda_v(D)=\max_{0<i\leq d} \frac{1}{d-i}\log^+\left(\frac{\|c_i\|}{\|c_d\|}\right).\end{equation}
\end{enumerate}
\end{lemma}

\begin{proof}
Property~\ref{it:linear} is immediate from the definition, while property~\ref{it:lambdamu} follows from standard properties of the Gauss norm and Newton Polygons. Specifically, let $i$ be an index which obtains the maximum in~\eqref{eq:lambdaGauss}. Then  $\|c_i\|$ is precisely the supremum of $c_i(x_0, ..., x_{N-1})$ on the unit multi-disk, which is realized at some point $(\beta_0, ..., \beta_{N-1})$ with $|\beta_i|=1$ for all $i$. Now, the quantity on the right of \eqref{eq:lambdaGauss} is the size of the largest root of the reciprocal polynomial to $F_D(\beta_0, ..., \beta_{N-1}, x_N)$, showing that $\lambda_v(D)$ is at least as large at this quantity. On the other hand, it follows from the definition of $\lambda_v(D)$ and the theory of Newton Polygons that
\begin{equation}\label{eq:lambdaGauss2}\lambda_v(D)=\sup_{|\beta_0|=|\beta_1|=\cdots=|\beta_{N-1}|=1}\max_{0<i\leq d} \frac{1}{d-i}\log^+\left(\frac{|c_i(\beta_0, ..., \beta_{N-1})|}{|c_d(\beta_0, ..., \beta_{N-1})|}\right).\end{equation}
But, given our normalization of $F_D$, we have $|c_d(\beta_0, ..., \beta_{N-1})|=1$ when $|\beta_0|=\cdots=|\beta_{N-1}|=1$, and so the right side of \eqref{eq:lambdaGauss2} is bounded above by the right side of \eqref{eq:lambdaGauss}.
\end{proof}

\begin{lemma}\label{lem:localtransform}
Let $f\in \pow{N}{d}(K)$, and let $D\in\Div^*(\PP^N)$.
\begin{enumerate}
\item $\lambda_v(f_*(D))\leq d\max\{\ph_v(f), \lambda_v(D)\}$
\item If $\lambda_v(D)>\ph(f)$, then we have
\[\lambda_v(f_*(D))=d\lambda_v(D).\]
\end{enumerate}
\end{lemma}

\begin{proof}
First note that if $-\log|\beta_N|>\ph_v(f)$, then we have $|\beta_N^{-1}|>|a_{i, I}|^{1/I_N}$ for each coefficient $a_{i, I}$ of $f$, whence $|a_{i, I}\beta_N^{I_N}|<1$. It follows that for any $\beta_0, ..., \beta_{N-1}\in K$, and for each $i$, if we set $\|\beta\|=\max\{|\beta_0|, ..., |\beta_{N-1}|\}$ then we have
\begin{equation}\label{eq:ultrametric}\left|\sum_{I_N\neq 0} a_{i, I}\beta_0^{I_0}\cdots\beta_N^{I_N}\right|<\|\beta\|^\kappa,\end{equation}
where $\kappa=d-1$ if $\|\beta\|\geq 1$, and $\kappa=1$ otherwise.

Now suppose that $\lambda_v(f_*(D))> d\max\{\ph_v(f), \lambda_v(D)\}$, so that there exist values $\alpha_0, ..., \alpha_N\in K$ with $[\alpha_0:\cdots :\alpha_N]\in f_*(D)$ such that $|\alpha_0|=\cdots=|\alpha_{N-1}|=1$, and $-\log|\alpha_N|=\lambda_v(f_*(D))>-d\max\{\ph_v(f), \lambda_v(D)\}$. If we choose $\beta_0, ..., \beta_N\in K$ with $f_i(\beta_0, ..., \beta_N)=\alpha_i$, and $\beta_N^d=\alpha_N$,  then
\[-\log|\beta_N|=\frac{-1}{d}\log|\alpha_N|>\max\{\ph_v(f), \lambda_v(D)\}.\] It follows that
\begin{equation}\label{eq:moreultra}\left|\alpha_i-\beta_i^d\right|=\left|\sum_{I_N\neq 0} a_{i, I}\beta_0^{I_0}\cdots\beta_N^{I_N}\right|<\max\left\{\|\beta\|, \|\beta\|^{d-1}\right\},\end{equation}
for each $0\leq i<N$. If $\|\beta\|> 1$, then we may choose $i$ with $|\beta_i|> 1$, and apply \eqref{eq:moreultra} to contradict the fact that $|\alpha_i|=1$.  So it must be that $\|\beta\|\leq 1$, and hence \eqref{eq:moreultra} implies that for each $i$ we have $|\alpha_i-\beta_i^d|<1$. But this is a contradiction if $|\beta_i|<1$ for any $i$, so in fact we have $|\beta_i|=1$ for all $i$. In other words, the point $[\beta_0:\cdots :\beta_N]$ witnesses
\[\lambda_v(D)\geq -\log|\beta_N|>\max\{\ph_v(f), \lambda_v(D)\}.\]
This is clearly a contradiction, and so the first claim is established.

It remains to show that if $\lambda_v(D)>\ph_v(f)$, then $\lambda_v(f_*(D))\geq d\lambda_v(D)$, so suppose that the hypothesis obtains. Then we have some values $\beta_0, ..., \beta_N\in K$ with $|\beta_0|=\cdots= |\beta_{N-1}|=1$ and $-\log|\beta_N|>\ph_v(f)$. Applying \eqref{eq:ultrametric}, we see that
\[\left|f_i(\beta_0, ..., \beta_N)\right|=\left|\beta_i^d+\sum_{I_N\neq 0} a_{i, I}\beta_0^{I_0}\cdots\beta_{N-1}^{I_{N-1}}\beta_N^{I_N}\right|=1\]
for each $i$. In other words, the point $[\alpha_0:\cdots :\alpha_{N}]\in f_*(D)$ defined by $\alpha_i=f_i(\beta_0, ..., \beta_N)$ and $\alpha_N=\beta_N^d$ 
 witnesses $\lambda_v(f_*(D))\geq d\lambda_v(D)$.
 \end{proof}

Lemma~\ref{lem:localtransform} is what enables us to define a local canonical height on divisors.
Let $f\in\pow{N}{d}(K)$ and let $D\in\Div^*_{K}(\PP^N)$. We will define a function $G_{f, v}$ by
\[G_{f, v}(D)=\lim_{n\to\infty} d^{-n}\lambda_v\left(f^n_*(D)\right).\]
Note that, in the case $N=1$, this definition reduces to the maximum value,  on the support of $D$, of the usual $v$-adic Green's function associated to $f$.

\begin{lemma}\label{lem:nonarchgreensprops}
\begin{enumerate}
\item\label{it:nongreensdefined} The function $G_{f, v}$ presented above is always defined.
\item\label{it:nongreenstransform} For any $D\in\Div^*_K(\PP^N)$, $G_{f, v}(f_*(D))=dG_{f, v}(D)$.
\item\label{it:nongreensper} If $D\in\Div^*_K(\PP^N)$ is preperiodic for $f\in \pow{N}{d}(K)$, then $G_{f, v}(D)=0$.
\item\label{it:nongreensbott} For any divisor $D\in\Div^*_K(\PP^N)$ satisfying $\lambda_v(D)>\ph_v(f)$, we have
\[G_{f, v}(D)=\lambda_v(D).\]
\item\label{it:nongreensasym} For any  $D\in\Div^*_K(\PP^N)$, we have
\[G_{f, v}(D)=\lambda_v(D)+O(\ph_v(f)),\]
where the implied constant is at most 2.
\end{enumerate}
\end{lemma}

\begin{proof} 
Property~\eqref{it:nongreensbott} follows from Lemma~\ref{lem:localtransform}. Specifically, if $\lambda_v(D)>\ph_v(f)$, then \[\lambda_v(f_*(D))=d\lambda_v(D)>\lambda_v(D)>\ph_v(f).\] Iterating this, we see that $\lambda_v(f^n_*(D))=d^n\lambda_v(D)$, whence follows the equality.

Property~\eqref{it:nongreensdefined} follows from property~\eqref{it:nongreensbott} if $\lambda_v(f^n_*(D))>\ph_v(f)$ for any $n$, while if this does not occur the definition gives $G_{f, v}(D)=0$. The definition of $G_{f,v}$ immediately gives property~\eqref{it:nongreenstransform}, once we know that the limit exists, while item~\eqref{it:linear} of Lemma~\ref{lem:lambdapropsnonarch} implies property~\eqref{it:nongreensper}. Specifically, it follows from Lemma~\ref{lem:lambdapropsnonarch} and the linearity of $f_*$ that \[G_{f, v}\left(\sum e_iD_i\right)= \max\{G_{f, v}(D_i)\}.\]
In particular, the values of $G_{f, v}$ are bounded on any preperiodic orbit, and hence by property~\eqref{it:nongreenstransform}
must vanish.

To prove property~\eqref{it:nongreensasym}, suppose first that $\lambda_v(f^m_*(D))\leq \ph_v(f)$ for all $m$. Then we have $G_f(D)=0$ and $\lambda_v(D)\leq\ph_v(D)$, implying the inequality. If, on the other hand, there is some $m\geq 0$ such that $\lambda(f^m_*(D))>\ph_v(f)$, then let $m$ be the least such value. In light of property~\eqref{it:nongreensbott}, we may as well assume that $m\geq 1$. We have
\[G_{f, v}(D)=d^{-m}G_{f, v}(f^m_*(D))=d^{-m}\lambda_v(f^m_*(D)),\] and since $\lambda_v(f^{m-1}_*(D))\leq \ph_v(f)$, we have by Lemma~\ref{lem:localtransform} that
$\lambda_v(f^m_*(D))\leq d\ph_v(f)$. It follows that
\[\left|\lambda_v(D)-G_{f, v}(D)\right|=\left|\lambda_v(D)-d^{-m}\lambda_v(f^m(D))\right|\leq \left(1+d^{1-m}\right)\ph_v(f)\leq 2\ph_v(f).\]
\end{proof}

Our next lemma estimates the value $\lambda_v(f_*(C_f))$. In some sense, the lemma is entirely standard given what we have already shown. Specifically, we have shown that the construction $f\mapsto f_*(C_f)$ is represented by rational function from one weighted projective space to another, which turns out by Lemma~\ref{lem:crux} to be a morphism of degree $d$. Since $\ph_v$ is the local height on the domain, and $\lambda_v$ is the local height on the range, one should expect that $\lambda_v(f_*(C_f))=d\ph_v(f)+O(1)$. Indeed, this is what the next lemma shows, and although the argument is standard, its application in a weighted projective space is possibly less familiar. For convenience of the reader we have included a proof.

\begin{lemma}\label{lem:localnull}
For any $f\in \pow{N}{d}(K)$ we have
\[\lambda_v(f_*(C_f))=d\ph_v(f)+O(1),\]
where the implied constant is absolute. Moreover, the implied constant vanishes unless $v$ is $p$-adic, for some $p\leq d$.
\end{lemma}

\begin{proof}
Note that for any monomial $c a_{i_1, I_i}\cdots a_{i_r, I_r}\in A=R[a_{i, I}]$, we have
\[|x|_v= |c|_v\cdot |a_{i_1, I_1}|_v\cdots |a_{i_r, I_r}|_v\leq |c|_v\max\{|a_{i_j, I_j}|^{1/I_{j, N}}\}^{I_{1, N}+\cdots +I_{r, N}}\]
and so for any homogeneous element $x\in A$ of degree $w$ we have
\begin{equation}\label{eq:homogupper}\log|x|_v\leq w\log^+\max\left\{|a_{i, I}|^{1/I_N}\right\}=w\ph_v(f)+O(1),\end{equation}
where the implied constant depends on $x$. Note, though, that this constant vanishes if the ring $R$ is contained in the ring of $v$-adic integers, that is, if $v$ is not $p$-adic for $p\leq d$.

Writing the defining equation of $f_*(C_f)$ as in \eqref{eq:branchdefinition}, we have
\[\log|b_J|_v\leq dJ_n\ph_v(f)+O(1),\]
and hence by property~\ref{it:lambdamu} of Lemma~\ref{lem:lambdapropsnonarch}, we have
\[\lambda_v(f_*(C_f))\leq d\ph_v(f)+O(1),\]
where again the constant vanishes unless $v$ is $p$-adic for $p\leq d$.

On the other hand, by Lemma~\ref{lem:nullstellensatz} we have an integer $e$ divisible by $(d-1)!$ and polynomials $g_{i, I, J}\in A$ homogeneous of degree $e-d!$ with respect to the gradation on $A$, such that for each
$0\leq i< N$ and multi-index $|I|=d$ with $I_N\neq 0, d$
 we have
\[a_{i, I}^{e/I_N}=\sum b_{J}^{(d-1)!/J_N}g_{i, I, J},\]
where the sum is over  multi-indices
$J$ with $|J|=d^{N-1}(d-1)N$.
The ultra-metric inequality then gives
\begin{eqnarray*}
e\log|a_{i, I}|^{1/N}&\leq& (d-1)!\log\max\{|b_J|^{1/J_N}\}+\log\max\{|g_{i, I, J}|\}\\
&\leq & (d-1)!\lambda_v(f_*(C_f))+(e-d!)\ph_v(f)+O(1),
\end{eqnarray*}
by Lemma~\ref{lem:lambdapropsnonarch} and \eqref{eq:homogupper}. Since this is true for each $0\leq i<N$ and each $I$, we have
\[e\ph_v(f)\leq (d-1)!\lambda_v(f_*(C_f))+(e-d!)\ph_v(f)+O(1),\]
and hence
\[d\ph_v(d)\leq \lambda_v(f_*(C_f))+O(1).\]
As above, the error term comes from the coefficients of the $g_{i, I, J}$, and hence vanishes of those are all $v$-integral, for instance of $v$ is not $p$-adic for any $p\leq d$.
\end{proof}

Finally, we come to a key estimate which is the main non-archimedean contribution to the results in this paper.
We define a function \[\lambda_{\mathrm{crit}, v}:\pow{N}{d}(K)\to\RR\] by $\lambda_{\mathrm{crit}, v}(f)=G_{f, v}(C_f)$. The key observation is that, by Lemma~\ref{lem:nonarchgreensprops}, the quantity $\lambda_{\mathrm{crit}, v}(f)$ vanishes if $f$ is PCF.

\begin{lemma}\label{lem:lyapunovlocal}
For $f\in \pow{N}{d}(K)$, we have
\[\lambda_{\mathrm{crit}, v}(f)=\ph_v(f)+O(1),\]
where the error term depends only on $N$, $d$, and $v$. Furthermore, unless $v$ is  $p$-adic for some $p\leq d$, the error term vanishes.
\end{lemma}

\begin{proof}
By Lemma~\ref{lem:localnull} we have
\[\lambda_v(f_*(C_f))=d\ph_v(f)+O(1),\]
and so if $(d-1)\ph_v(f)$ is larger than the implied constant, Lemma~\ref{lem:nonarchgreensprops} gives us
\[\lambda_{\mathrm{crit}}(f)=\frac{1}{d}G_{f, v}(f_*(C_f))=\frac{1}{d}\lambda(f_*(C_f))=\ph_v(f)+O(1).\]
If $(d-1)\ph_v(f)$ does not exceed the constant, then $\ph_v(f)=O(1)$, with a constant depending just on $d$ and $N$, and so $\lambda_v(f_*(C_f))=O(1)$.  Lemma~\ref{lem:nonarchgreensprops} now gives us that $G_{f, v}(f_*(C_f))=O(1)$, and hence $\lambda_{\mathrm{crit}}(f)=O(1)$.

If $v$ is not $p$-adic for some $p\leq d$, then all of the implied constants vanish.
\end{proof}

We may now prove some of the main results.
\begin{proof}[Proof of Theorem~\ref{th:rigidity}]
Let $k$ be an algebraically closed field, and suppose that there is a non-constant rational map $\phi:X\to \pow{N}{d}$ over $k$ whose image lands entirely in one of the varieties $C_{n, m}$. This map corresponds to a point in $C_{n, m}$ with coordinates in the function field $k(X)$. Let $\beta\in X(k)$ be any point, and let $K$ be the local field of functions on $X$ at $\beta$, with the usual absolute value $|\cdot|_v$. Since this is not a $p$-adic absolute value, and since $\lambda_{\mathrm{crit}, v}(f)=0$, we have $\ph_v(f)=0$. Note that, in geometric terms, this means that none of the coefficients $a_{i, I}\in k(X)$ defining $f$ have a pole at $\beta$. Since $\beta$ was arbitrary, the coefficients $a_{i, I}\in k(X)$ are regular on all of $X$, and hence are constant. This contradicts our hypothesis that $\phi:X\to \pow{N}{d}$ was non-constant. So the varieties $C_{n, m}$ are 0-dimensional, and hence so are the images $\mathsf{C}_{n, m}\subseteq \powm{N}{d}$. This finishes the proof of Theorem~\ref{th:rigidity}.
\end{proof}

\begin{proof}[Proof of Theorem~\ref{th:goodred}]
Assume that $K$ is not $p$-adic, for any $p\leq d$. Then by the lemmas above, we have $\lambda_{\mathrm{crit}, v}(f)=\ph_v(f)$ for every $f\in\pow{N}{d}$. If $f$ is PCF, then we must have $\lambda_{\mathrm{crit}, v}(f)=0$ by Lemma~\ref{lem:nonarchgreensprops}, and hence $|a_{i, I}|\leq 1$ for each $i$ and $I$. Since the coefficient of $x_i^d$ in the $i$th coordinate of $f$ is a $p$-adic unit, this means that $f$ has good reduction.
\end{proof}

\begin{proof}[Proof of Corollary~\ref{cor:repelling}]
By the previous claim, the coefficients of $f$ are all integral at $p$. In particular, the Jacobian of the fixed point $\origin$ of $f$ has integral entries, and so its characteristic polynomial has integral roots. Since the coordinates of every fixed point must be integral, a change of coordinates shows that the eigenvalues of the Jacobian matrices at the other fixed points are integral as well.  Applying the argument to $f^n$ proves the claim for $n$-periodic points.
\end{proof}

%%%%%%%%%%%%%%%%%%%%%%%%%%%%%%%%%%%%%%%%%%
%%%%%%%%%%%%%%%%%%%%%%%%%%%%%%%%%%%%%%%%%%
%%%%%%%%%%%%%%%%%%%%%%%%%%%%%%%%%%%%%%%%%%
%%%%%%%%%%%%%%%%%%%%%%%%%%%%%%%%%%%%%%%%%%
%%%%%%%%%%%%%%%%%%%%%%%%%%%%%%%%%%%%%%%%%%
%%%%%%%%%%%%%%%%%%%%%%%%%%%%%%%%%%%%%%%%%%
%%%%%%%%%%%%%%%%%%%%%%%%%%%%%%%%%%%%%%%%%%
%%%%%%%%%%%%%%%%%%%%%%%%%%%%%%%%%%%%%%%%%%

\section{Archimedean places}\label{sec:complex}

In this section we work entirely over the field $\CC$ of complex numbers, with the usual absolute value. In general, this section closely parallels the structure of Section~\ref{sec:local}, but the estimates are more involved, since the absolute value is not ultrametric.
Let $D\in\Div^*_\CC(\PP^N)$ be a divisor defined by a homogeneous form $F_D$ normalized so that
$F_D(x_0, ..., x_{N-1}, 0)=\prod x_i^{e_i}$.
 We let $S\subseteq \CC^{N}$ denote the set
\[S=\left\{(x_0, ..., x_{N-1})\in \CC^{N}:|x_0|=\cdots =|x_{N-1}|=1\right\},\]
and define
\[\lambda_\infty(D)=\sup_{(\beta_0, ..., \beta_{N-1})\in S}\log^+\max\left\{|\beta_N|^{-1}:F_D(\beta_0, ..., \beta_N)=0\right\}.\]
Note that, for each point $(\beta_0, ..., \beta_{N-1})\in S$, the possible values $\beta_N$ satisfy a polynomial equation with a non-zero constant term, and so in particular are non-zero.
Since it is the supremum of a continuous function on a compact set, the value $\lambda_\infty(D)$ is witnessed by at least one point. 
We note that it is immediate from the definition of $\lambda_\infty$, just as in Section~\ref{sec:local}, that \[\lambda_\infty(D+E)=\max\{\lambda_\infty(D), \lambda_\infty(E)\},\]
 and in particular that $\lambda_\infty(f(D))=\lambda_\infty(f_*(D))$ for any $D\in \Div^*(\PP^N)$.
 
 Just as in Section~\ref{sec:global}, when $f\in\pow{N}{d}(\CC)$ is written as in \eqref{eq:genericpoly} we will define
\[\ph_\infty(f)=\log^+\max\{|a_{i, I}|^{1/I_N}:0\leq i<N, \text{ and }I\in \Ind^*(N, d)\}.\]

\begin{lemma}\label{lem:annulus}
Let $D\in\Div^*(\PP^N)$ be the divisor associated to $F_D$, let
\[S_r=\left\{(x_0, ..., x_{N-1})\in \CC^{N}:r^{-1}\leq|x_0|,\cdots ,|x_{N-1}|\leq r\right\},\]
and set
\[\lambda_\infty(D; r)=\sup_{(\beta_0, ..., \beta_{N-1})\in S_r}\log^+\max\left\{|\beta_N|^{-1}:F_D(\beta_0, ..., \beta_N)=0\right\}.\]
Then
\[\lambda_\infty(D; r)-\log r\leq \lambda_\infty(D) \leq \lambda_\infty(D; r).\]
\end{lemma}

\begin{proof}
Note that the second inequality is trivial, since $S=S_1\subseteq S_r$.

 As usual, let $F_D$ define $D$.
We construct an affine variety $X/\CC$ with a map \[\mathbf{x}=(x_0, ..., x_{N-1}):X\to\CC^N\] and functions $\beta, \zeta_i\in\Ocal(X)$ for $1\leq i\leq \deg(D)$ such that 
\[\beta^{\deg(F_D)}=F_D(x_0, ..., x_{N-1}, 0)=\prod_{i=0}^{N-1}x_i^{e_i}\] and
\[F_D(Yx_0, ..., Yx_{N-1}, 1)=\prod_{i=1}^{\deg(D)}(\beta Y -\zeta_i).\]
%For instance,
%\[X=\Spec\frac{\CC[x_0, ..., x_{N-1}, \beta, \zeta_1, ..., \zeta_{\deg(D)}, Y]}{(\beta^{\deg(D)}-F_D(x_0, ..., x_{N-1}, 0), \beta^{\deg(D)(\deg(D)-1)}F(\beta^{-\deg(D)}Yx_0, ..., \beta^{-\deg(D)}Yx_{N-1}, 1))}\]
%\nts{clear this up...}
If $U\subseteq X$ is the largest affine open set with $x_i^{-1}\in \Ocal(U)$ for all $i$, then $\beta^{-1}\in \Ocal(U)$ as well. Note that the map $U\to \CC^N$ defined by $(\zeta_i x_0/\beta, ..., \zeta_i x_{N-1}/\beta)$ lands entirely in $D$, and that every point in $D$ is in the image of one of these maps. The pull-back to $X$ of $D$ is simply the sum of $\beta-\zeta_i=0$.

Now, let (for $r\leq R$ non-zero)
\[A(r, R)=\{P\in X: r \leq  |x_i(P)| \leq R\text{ for all }0\leq i<N\},\]
a compact subset of $U\subseteq X$,
and
\[B=\{P\in X:  |x_i(P)| \leq 1\text{ for all }0\leq i<N\}.\]
For any set $Z$, let $\|\cdot\|_{Z}$ to be the sup norm on $Z$.
We have, by definition, $\lambda_\infty(D)=\lambda_\infty(D; 1)$, and
\[\lambda_\infty(D; r)=\log^+\max_{1\leq i\leq \deg(D)}\|\zeta_i/\beta\|_{A(r^{-1}, r)}.\]
Now, since $A(1, 1)\subseteq A(r^{-1}, r)$, the inequality $\lambda_\infty(D)\leq \lambda_\infty(D; r)$ is trivial. On the other hand, note that since $|\beta|=1$ identically on $A(1, 1)$, we have
\[\|\zeta_i/\beta\|_{A(1, 1)}=\|\zeta_i\|_{A(1)}=\|\zeta_i\|_{B}\]
for any $i$, by the maximum principle. On the other hand, $A(r^{-2}, 1)\subseteq B$, and so $\|\zeta_i\|_{A(r^{-2}, 1)}\leq \|\zeta_i\|_B$. Now note that the $\zeta_i$ are invariant under the action of $\mathbb{G}_\mathrm{m}$ corresponding to $(x_0, ..., x_{N-1})\mapsto (\alpha x_0, ..., \alpha x_{N-1})$, and so we have
\[\|\zeta_i\|_{A(r^{-2}, 1)}=\|\zeta_i\|_{A(r^{-1}, r)}.\]
Since $|\beta|\geq r^{-1}$ on $A(r^{-1}, r)$, we have $r^{-1}\|\zeta_i/\beta\|_{A(r^{-1}, r)}\leq \|\zeta_i/\beta\|_{A(1, 1)}$, proving the result.
\end{proof}

The next lemma shows us that when $\lambda_\infty(D)$ is large relative to $\ph_\infty(f)$, the quantity $\lambda_\infty(f_*(D))$ is relatively predictable.

\begin{lemma}\label{lem:complextransform}
For any $D\in \Div^*(\PP^N)$, we have
\[
\lambda_\infty(f_*(D))\leq d\max\left\{\lambda_\infty(D), \ph_\infty(f) +\log(2\dim(\pow{N}{d})/N)\right\}-\log(1-2^{-1/d}).\]
Furthermore, if
\[\lambda(D)>\ph(f)+\log\left(\frac{2\dim(\pow{N}{d})}{N}\right),\]
then  
\begin{equation}\label{eq:comptransmain}d\lambda(D)- \log 2\leq \lambda(f_*(D))\leq d\lambda(D)+\log\left(\frac{1}{1-2^{-1/d}}\right).\end{equation}
\end{lemma}

\begin{proof}
Consider any values $\beta_0, ..., \beta_N\in \CC$ with  $[\beta_0: \cdots: \beta_N]\in D$ and 
\begin{equation}\label{eq:beta}-\log|\beta_N|>\ph_\infty(f)+\log(2\dim(\pow{N}{d})/N).\end{equation}
Note that $|a_{i, I}\beta_N|^{I_N}<\frac{2N}{\dim(\pow{N}{d})}$ for each $i$ and $I$, and so if we set \[\|\beta\|=\max\{|\beta_1|, ..., |\beta_{N-1}|\}\] we have
\begin{eqnarray}
\left|f_i(\beta_0, ..., \beta_N)\right|&=&\left|\sum_{I_N\neq 0, d}a_{i, I}\beta_0^{I_0}\beta_1^{I_1}\cdots\beta_N^{I_N}\right|\nonumber\\
&\leq&\frac{1}{2}\max\{\|\beta\|, \|\beta\|^{d-1}\}\label{eq:triangle}
\end{eqnarray}
by the triangle inequality, since $\dim(\pow{N}{d})/N$ is precisely the number of summands.

Suppose that
\[
\lambda_\infty(f_*(D))> d\max\left\{\lambda_\infty(D), \ph_\infty(f) +\log(2\dim(\pow{N}{d})/N)\right\}-\log(1-2^{-1/d}).\]
Then there is some point \[Q=[\alpha_0: \cdots :\alpha_{N}]\in f_*(D)\] with $|\alpha_i|=1$ for all $0\leq i<N$, and $-\log|\alpha_N|=\lambda_\infty(f_*(D))$.
Choose a point $P=[\beta_0: \cdots : \beta_N]\in D$ so that $Q=f(P)$, with representative coordinates chosen so that $\alpha_N=\beta_N^d$, noting that this ensures that \eqref{eq:beta} holds.

 If $\|\beta\|>2^{1/d}$ then choosing $i$ with $|\beta_i|$ maximal gives
 \[1=|\alpha_i|=\left|\beta_i^d+\sum_{I_N\neq 0, d}a_{i, I}\beta_0^{I_0}\cdots\beta_N^{I_N} \right|\geq |\beta_i|^d-\frac{1}{2}\|\beta\|^{d-1}>1,\]
 by~\eqref{eq:triangle}, a contradiction. So we have $\|\beta\|\leq 2^{1/d}$.

 If $\|\beta\|\geq 1$, then~\eqref{eq:triangle} gives for any $j$
 \[|\beta_j^d|=\left|\alpha_j-\sum_{I_N\neq 0, d}a_{i, I}\beta_0^{I_0}\cdots\beta_N^{I_N}\right|\geq 1-\frac{1}{2}(2^{1/d})^{d-1}=1-2^{-1/d}.\]
 This gives 
 \[(1-2^{-1/d})^{1/d}\leq |\beta_j|\leq 2^{1/d}\leq (1-2^{-1/d})^{-1/d},\]
 and so
 $(\beta_0, .., \beta_{N-1})\in S_r$ for $r=(1-2^{-1/d})^{-1/d}$.
From this we see by Lemma~\ref{lem:annulus} that
\begin{multline*}\max\left\{\lambda_\infty(D), \ph_\infty(f) +\log(2\dim(\pow{N}{d})/N)\right\}-\frac{1}{d}\log(1-2^{-1/d})\\<\lambda_\infty\left(D; (1-2^{-1/d})^{-1/d}\right)\leq \lambda_\infty(D)-\frac{1}{d}\log (1-2^{-1/d}),\end{multline*}
 a contradiction.

It must then be the case that $\|\beta\|<1$, from which we have, for each $0\leq j<N$,
\[1=|\alpha_j|\leq |\beta_j|^d+\frac{1}{2}\|\beta\|<|\beta_j|^d+\frac{1}{2}\]
from \eqref{eq:triangle}. It follows from this that $|\beta_j|\geq 2^{-1/d}$ for each $0\leq j<N$, and so
 $(\beta_0, ..., \beta_{N-1})\in S_{2^{1/d}}$. Again applying Lemma~\ref{lem:annulus},
\begin{multline*}\max\left\{\lambda_\infty(D), \ph_\infty(f) +\log(2\dim(\pow{N}{d})/N)\right\}-\frac{1}{d}\log(1-2^{-1/d})\\<\lambda_\infty(D; 2^{1/d})\leq 
\lambda_\infty(D)+\frac{1}{d}\log 2.\end{multline*}
 This is again a contradiction, since $(1-2^{-1/d})^{-1}>2$ for all $d\geq 2$, and so we have verified the first claim in the lemma.
Note that this also verifies the upper bound in~\eqref{eq:comptransmain}.

Now, if \[\lambda_\infty(D)>\ph_\infty(f)+\log(2\dim(\pow{N}{d})/N),\] there is a point $P=[\beta_0:\cdots:\beta_N]\in D$ as above such that $|\beta_i|=1$ for all $i$, and $-\log|\beta_N|=\lambda_\infty(D)$. By Equation~\eqref{eq:triangle}, we have
\[\left|f_i(\beta_0, ..., \beta_N)-\beta_i^d\right|\leq \frac{1}{2},\]
and hence
\[\frac{1}{2}\leq |f_i(\beta_0, ..., \beta_N)|\leq \frac{3}{2}.\]
Thus if $\alpha_i=f_i(\beta_0, ..., \beta_N)$ for each $0\leq i<N$, and $\alpha_N=\beta_N^d$, then the point $[\alpha_0: \cdots:\alpha_{N}]\in f_*(D)$ demonstrates the inequality 
\[d\lambda_\infty(D)\leq \lambda_\infty(f_*(D); 2)\leq \lambda_\infty(f_*(D))+\log 2\]
by Lemma~\ref{lem:annulus}. This verifies the lower bound in \eqref{eq:comptransmain}
\end{proof}

In light of Lemma~\ref{lem:complextransform}, we define for $D\in\Div^*(\PP^N)$
\[G_{f, \infty}(D)=\lim_{n\to\infty} d^{-n}\lambda(f_*^n(D)).\]
The following properties are deduced easily.
\begin{lemma}\label{lem:complexgreens}
\begin{enumerate}
%\item $G_f(mD)=G_f(D)$ for any divisor $D$ and any integer $m\geq 1$.
%\item $G_f(D+E)\leq \max\{G_f(D), G_f(E)\}$ for any divisors $D, E$.
\item $G_{f, \infty}(f_*(D))=dG_{f, \infty}(D)$ for any divisor $D\in\Div^*(\PP^N)$.
\item $G_{f, \infty}(D)=0$ for any $f$-preperiodic divisor $D\in\Div^*(\PP^N)$.
\item\label{it:complexbott} If $\lambda_\infty(D)>\ph_\infty(f)+\log\left(\frac{2\dim(\pow{N}{d})}{N}\right)$, then \[G_{f, \infty}(D)=\lambda_\infty(D)+O(1),\] where the implied constant depends only on $d$.
\item\label{it:complexasym} For any  $D\in\Div^*_\CC(\PP^N)$, we have
\[G_{f, \infty}(D)=\lambda_\infty(D)+O(\ph_\infty(f)),\]
where the implied constant depends only on $d$ and $N$.
\end{enumerate}
\end{lemma}

\begin{proof}
First we note that the limit always exists. In particular, by Lemma~\ref{lem:complextransform} if 
\begin{equation}\label{eq:greensassump}\lambda_\infty(D)>\ph_\infty(f)+\log\left(\frac{2\dim(\pow{N}{d})}{N}\right),\end{equation}
 we have
\[\left|d^{-1}\lambda_\infty(f_*(D))-\lambda_\infty(D)\right|\leq \frac{1}{d}\log\left(\frac{1}{1-2^{-1/d}}\right).\]
By the triangle inequality, for any $n$ we then have
\begin{eqnarray}
\left|d^{-n}\lambda_\infty(f_*^n(D))-\lambda_\infty(D)\right|&\leq& \left(\frac{1}{d}+\cdots +\frac{1}{d^n}\right) \log\left(\frac{1}{1-2^{-1/d}}\right)\nonumber\\
&\leq & \frac{1}{d-1}\log\left(\frac{1}{1-2^{-1/d}}\right).\label{eq:complexgreensbound}
\end{eqnarray}
Replacing $D$ with $f_*^m(D)$ shows that the sequence whose limit defines $G_{f_\infty}(D)$ is Cauchy under the hypothesis of~\eqref{eq:greensassump}. This shows that $G_{f, \infty}(D)$ is defined whenever $f_*^m(D)$ satisfies \eqref{eq:greensassump} for some $m\geq 0$, and of course if this never happens the definition gives $G_{f, \infty}(D)=0$.
Equation~\eqref{eq:complexgreensbound} also gives the fifth property, with an explicit constant (which turns out, in fact, to be absolute).

The fourth property follows from those prior, while the third property is an immediate consequence of the definition of $G_f$. The first two properties follow from an inspection of the definition of $\lambda_\infty$.

To prove the sixth property, suppose that there is some $m\geq 0$ such that $\lambda_\infty(f^m(D))>\ph_\infty(f)$, and let $m$ be the least such value. Given property (5), we may as well assume that $m\geq 1$. We have $G_{f, \infty}(D)=d^{-m}\lambda_\infty(f^m(D))+O(1)$, where the implied constant is absolute, and since both $\lambda_\infty(D)$ and $\lambda_\infty(f^{m-1}(D))$ are bounded above by $\ph_\infty(f)+\log\left(\frac{2\dim(\pow{N}{d})}{N}\right)$, we have by Lemma~\ref{lem:complextransform}
\[\lambda(f^m(D))\leq d\ph_v(f)+O(1),\] and hence
\begin{eqnarray*}
\left|\lambda_\infty(D)-G_{f, \infty}(D)\right|&=&\left|\lambda_\infty(D)-d^{-m}\lambda_\infty(f^m(D))\right|\\
&\leq& \left(1+d^{1-m}\right)\ph_\infty(f)+O(1)\\
&\leq& 2\ph_\infty(f)+O(1).
\end{eqnarray*}
If, on the other hand, $\lambda_\infty(f^m(D))\leq \ph_\infty(f)+O(1)$ for all $m$, we have $G_{f, \infty}(D)=0$ and $\lambda_\infty(D)\leq\ph_\infty(f)+O(1)$, again implying the inequality.
\end{proof}

\begin{lemma}\label{lem:lambdaandcoeffscomplex}
Let $D$ be a divisor as above, and suppose that $D$ is defined by $\sum b_I \mathbf{x}^I=0$. Then
\[\lambda_\infty(D)=\log\max|b_I|^{1/I_N}+O(1),\]
where the implied constant is no larger than $\log\deg(D)+1$.
\end{lemma}

In the proof of Lemma~\ref{lem:lambdaandcoeffscomplex} we will use the following easy estimate, essentially saying the Newton Polygon of a complex polynomial does a reasonable job of estimating the size of the largest root. The proof is straight foward and relatively standard, and so is omitted.
\begin{lemma}\label{lem:NMineq}
For any $g(z)\in\CC[z]$ with $\deg(g)\geq 1$, let 
\[N(g)=\max\left\{\left(\frac{|a_i|}{|a_{\deg(g)}|}\right)^{1/(\deg(g)-i)}:0\leq i < \deg(g)\right\}\]
and
\[M(g)=\max\left\{|\beta_N|:g(\beta_N)=0\right\}.\]
Then
\[\log M(g)-\log \deg(g)\leq \log N(g)\leq \log M(g)+\log \deg(g)+1.\]
\end{lemma}

%\begin{proof}
%Let $\delta=\deg(g)$ and, without loss of generality, suppose that $a_\delta=1$. First, suppose that $z\in \CC$ satisfies $|z|>\delta N(g)$. The for each $i$ we have
%\[|z|>\delta|a_i|^{1/(\delta-i)}\]
%whence
%\[|z^\delta|>\delta^{\delta-i}|a_iz^i|\geq \delta|a_iz^i|.\]
%It follows that
%\[|g(z)|\geq |z^\delta|-\left|\sum_{i=0}^{\delta}a_iz^i\right|>0,\]
%and in particular $g(z)\neq 0$.
%Applying the converse, we see that $M(g)\leq \delta N(g)$.
%
%On the other hand, note that if we list the roots of $g$ as $\xi_1, ..., \xi_\delta$, then
%\[a_i=(-1)^{\delta-i}\sum_{\substack{I\subseteq\{1, ..., \delta\}\\ \# I = \delta-i}}\prod_{i\in I} \xi_i.\]
%In particular,
%\[|a_i|\leq \binom{\delta}{\delta-i}M(g)^{\delta-i},\]
%whence
%\[|a_i|^{1/(\delta-i)}\leq \binom{\delta}{\delta-i}^{1/(\delta-i)}M(g)\leq\frac{e\delta}{\delta-i}M(g),\]
%using the standard estimate $\binom{n}{m}\leq (en/m)^m$. Taking the maximum over all $i$ yields the bound
%\[N(g)\leq e\delta M(g).\]
%\end{proof}

\begin{proof}[Proof of Lemma~\ref{lem:lambdaandcoeffscomplex}]
If $F=F_D$ is the form defining $D$, let \[F(x_0, ..., x_N)=\sum c_i(x_0, ..., x_{N-1})x_N^{i}.\]
 We note that if
\[\tilde{F}_{\alpha_0, ..., \alpha_{N-1}}(Y)=Y^{\deg(F)}F(\alpha_0, ..., \alpha_{N-1}, Y^{-1}),\]
then by definition
\[\lambda(D)=\sup_{\alpha\in S} \log\max\{|\beta_N|:\tilde{F}_{\alpha_0, .., \alpha_{N-1}}(\beta_N)=0.\}\]
For each $\alpha$,
\[\log M(\tilde{F}_{\alpha})\leq \log N(\tilde{F}_{\alpha})+\log\deg(F)\leq\log\max|b_I|^{1/I_N}+\log\deg(D), \]
whence
\[\lambda(D)\leq \log\max|b_I|^{1/I_N}+\log\deg(D).\]

On the other hand, for any polynomial \[g(z_0, ..., z_{N-1})=\sum \beta_I\mathbf{z}^I\in\CC[z_0, ..., z_{N-1}],\] if we denote by $\|g\|$ the supremum of on the unit multidisk, then Parseval's identity gives us
\begin{equation}\label{eq:parseval}\left(\max |\beta_I|\right)^2= \sum |\beta_I|^2=\int |g(\mathbf{z})|^2d\mu(\mathbf{z})\leq \|g\|^2,\end{equation}
where $\mu$ is the Haar measure on the $N$th power of the unit circle. If the index $J$ witnesses the maximum in $\max|b_I|^{1/I_N}$, then we may apply \eqref{eq:parseval} to the polynomial
\[c_{I_N}(x_0, ..., x_{N-1})=\sum_{I_N=J_N} b_I x_0^{I_0}\cdots x_{N-1}^{I_{N-1}}\]
to find a point $\alpha_0, ..., \alpha_{N-1}$ on the $N$th power of the unit circle such that
\[\log|c_{I_N}(\alpha_0, ..., \alpha_{N-1})|^{1/I_N}\geq \log |b_J|^{1/J_N}=\log\max|b_I|^{1/I_N}.\]
For $\alpha=(\alpha_0, ..., \alpha_{N-1})$, the above lemma then gives
\[\lambda(D)\geq \log M(\tilde{F}_\alpha)\geq \log N(\tilde{F}_\alpha)-\log\deg(D)-1\geq \log\max|b_I|^{1/I_N}-\log\deg(D)-1.\]
\end{proof}

Finally, we define a function $\lambda_{\mathrm{crit}, \infty}:\pow{N}{d}(\CC)\to\RR$ by
\[\lambda_{\mathrm{crit}, \infty}(f)=G_{f, \infty}(C_f).\]
\begin{lemma}\label{lem:complexcritlowerbound}
There is a constant $B$ with
\[\lambda_{\mathrm{crit}, \infty}(f)=\ph_\infty(f)+O(1),\]
where the implied constant depends just on $N$ and $d$.
\end{lemma}

\begin{proof}
The proof of this lemma is essentially the same as that of Lemma~\ref{lem:lyapunovlocal}.
Writing $f_*(C_f)$ as $\sum b_J \mathbf{x}^J=0$, we have
\[\lambda_\infty(f_*(C_f))= \log^+\max\{|b_J|^{1/J_N}\}+O(1)\]
by Lemma~\ref{lem:lambdaandcoeffscomplex}.
Note that $b_J\in A$ a homogeneous element of weight $dJ_N$, and by Lemma~\ref{lem:nullstellensatz}, there exists an integer $e$ divisible by $d!$ and polynomials $g_{i, I, J}\in A$, homogeneous of degree $d!-e$, such that
\[a_{i, I}^{e/I_N}=\sum b_{J}^{(d-1)!/J_N}g_{i, I, J}.\]
By the triangle inequality, we have
\begin{eqnarray*}
e\ph_v(f)&=&e\log^+\max\{|a_{i, I}|^{1/I_N}\}\\
&\leq &\log^+ \max\{ |b_{J}^{(d-1)!/J_N}g_{i, I, J}|\}+O_{N, d}(1)\\
&\leq &(d-1)!\log^+\max\{|b_J|^{1/J_N}\}+(e-d!)\log^+\max\{|a_{i, I}|^{1/I_N}\}+O_{N, d}(1)\\
&=&(d-1)!\log^+\max\{|b_J|^{1/J_N}\}+(e-d!)\ph_v(f)+O_{N, d}(1),
\end{eqnarray*}
where the implied constant incorporates the coefficients of each $g_{i, I, J}$ as polynomials in the $a_{i, I}$. Rearranging this gives
\begin{equation}\label{eq:nullineq}d\ph_\infty(f)\leq \log^+\max\{|b_J|^{1/J_N}\}+O_{N, d}(1)\leq\lambda_\infty(f_*(C_f))+O_{N, d}(1).\end{equation}

On the other hand, each $b_J$ is a weighted-homogeneous polynomial in the $a_{i, I}$ of degree $d$, from which it follows easily that
\[\log^+|b_J|^{1/J_N}\leq d\ph_\infty(f)+O(1).\]
This gives the complementary inequality to~\eqref{eq:nullineq}, showing that
\begin{equation}\label{eq:complexnull}\lambda_\infty(f_*(C_f))=d\ph_\infty(f)+O(1),\end{equation}
where the implied constant depends just on $N$ and $d$.

For $\ph_\infty(f)$ large enough, \eqref{eq:complexnull} and Lemma~\ref{lem:complexgreens} combine to give
\[G_{f, \infty}(C_f)=\frac{1}{d}G_{f, \infty}(f_*(C_f))=\frac{1}{d}\lambda_\infty(f_*(C_f))+O(1)=\ph_\infty(f)+O(1).\]
If $\ph_\infty(f)$ is not large enough for \eqref{eq:complexnull} to ensure that $f_*(C_f)$ meets the hypotheses of Lemma~\ref{lem:complexgreens} property~\eqref{it:complexbott}, then $\ph_\infty(f)$ is bounded by a quantity depending just on $d$ and $N$.  Applying property~\eqref{it:complexasym} of Lemma~\ref{lem:complexgreens}, we see that \eqref{eq:complexnull} implies $G_{f, \infty}(f_*(C_f))=O(\ph_\infty(f))$, with the implied constant absolute, and hence $\lambda_{\mathrm{crit}, \infty}(f)=G_{f, \infty}(C_f)$ is also bounded by a constant depending just on $N$ and $d$. This establishes the inequality in the lemma when $\ph_\infty(f)$ is bounded.
\end{proof}

Finally, we note that the function $\lambda_{\mathrm{crit}, \infty}$ is in fact continuous on $\pow{N}{d}$.
\begin{lemma}\label{lem:gfcts}
The function $\lambda_{\mathrm{crit}, \infty}:\pow{N}{d}(\CC)\to\RR$ defined by $f\mapsto G_{f, \infty}(C_f)$ is continuous.
\end{lemma}

\begin{proof}
First we note that the coefficients of $C_f$ are polynomials defined on $\pow{N}{d}(\CC)$, and hence are continuous, and by the continuity of roots of a polynomial in its coefficients, the function $f\mapsto \lambda_\infty(C_f)$ is continuous. In exactly the same fashion, each function $f\mapsto d^{-n}\lambda_\infty(f^n_*(C_f))$ is continuous. Although $G_{f, \infty}$ is defined to be the pointwise limit of the $d^{-n}\lambda_\infty(f_*^n(C_f))$, we will show that the convergence is in fact uniform.

Define
\[X_m=\left\{f\in\pow{N}{d}(\CC):\lambda_\infty(f^m(C_f))> \ph_\infty(f)+\log (2\dim(\pow{N}{d})/N)\right\}\subseteq \pow{N}{d}(\CC)\]
then the proof of Lemma~\ref{lem:complexgreens} shows that 
\[\left|d^{-m}\lambda_\infty(f_*^m(C_f))-G_{f, \infty}(C_f)\right|\leq d^{-m}c_1,\]
for some constant $c_1$ independent of $m$ and $f$. On the other hand, since $X_1\subseteq X_m$ for all $m$, the proof of Lemma~\ref{lem:complexcritlowerbound} shows that $\ph_\infty(f)$ is bounded on the complement of $X_m$, by a quantity which does not depend on $m$. Since
\[\lambda_\infty(D)= G_{f, \infty}(D)+O(\ph_\infty(f))\]
by Lemma~\ref{lem:complexgreens}
we have
\[\left|d^{-m}\lambda_\infty(f_*^m(C_f))-G_{f, \infty}(C_f)\right|\leq d^{-m}c_2\]
on the complement of $X_m$, where again the constant is independent of $m$ or $f$. This shows that $d^{-n}\lambda_\infty(f^n_*(C_f))\to G_{f, \infty}(C_f)$ uniformly on $\pow{N}{d}(\CC)$, and hence the function $\lambda_{\mathrm{crit}, \infty}:\pow{N}{d}(\CC)\to\RR$ is continuous.
\end{proof}

\begin{proof}[Proof of Theorem~\ref{th:compactness}]
By Lemma~\ref{lem:gfcts} we know that the set defined by $G_{f, \infty}(C_f)\leq B$ is closed, for each $B\in \RR$, and by Lemma~\ref{lem:complexcritlowerbound} we know that this set is bounded.
\end{proof}

Note that Theorem~\ref{th:compactness} already gives a proof of Theorem~\ref{th:rigidity} over $\CC$, but we offer a proof below that works more generally and involves less analytic baggage.

Before moving on to global fields, we note a possible connection with pluripotential theory. Associate to the map $f:\PP^N_\CC\to\PP^N_\CC$ is an invariant current $T_f$, and if $[C_f]$ is the current of integration along the critical divisor, Bedford and Jonsson~\cite{bj} consider the value
\begin{equation}\label{eq:bj}\int G_f^{\mathrm{points}} [C_f]\wedge T_f^{N-1},\end{equation}
where $G_f^{\mathrm{points}}:\CC^N\to \RR$ is the Green's function on points (denoted simply by $G_f$ in \cite{bj, pi:specN}). In particular, this value determines the Lyapunov exponent of any map $f\in \powm{N}{d}(\CC)$. It seems plausible that the quantity defined in \eqref{eq:bj} is closely related to the quantity $\lambda_{\mathrm{crit}, \infty}$, and establishing an explicit relation might provide a lower bound on the Lyapunov exponent on $\powm{N}{d}$ which is arbitrarily large off a compact set. In the one-dimensional case, for example, the quantity defined in \eqref{eq:bj} becomes $\sum_{f'(c)=0}G_f^{\mathrm{points}}(c)$, while $\lambda_{\mathrm{crit}, \infty}(f)$ becomes $\max_{f'(c)=0}G_f^{\mathrm{points}}(c)$. Since the critical divisor always has degree $d-1$, these are easily related.

%%%%%%%%%%%%%%%%%%%%%%%%%%%%%%
%%%%%%%%%%%%%%%%%%%%%%%%%%%%%%
%%%%%%%%%%%%%%%%%%%%%%%%%%%%%%
%%%%%%%%%%%%%%%%%%%%%%%%%%%%%%
%%%%%%%%%%%%%%%%%%%%%%%%%%%%%%
%%%%%%%%%%%%%%%%%%%%%%%%%%%%%%
%%%%%%%%%%%%%%%%%%%%%%%%%%%%%%
%%%%%%%%%%%%%%%%%%%%%%%%%%%%%%

\section{Global fields}\label{sec:global}

Throughout this section, $K$ will denote a number field, although much of the discussion pertains as well when $K$ is a function field of characteristic $0$ or $p>d$. We will let $M_K$ denote the usual set of places of $K$, and we will define
\begin{equation}\label{eq:weildef}h_{\mathrm{Weil}}(f)=\sum_{v\in M_K}\frac{[K_v:\QQ_v]}{[K:\QQ]}\ph_v(f)\end{equation}
noting that, by standard facts about the splitting of valuations in extensions, this definition is in fact independent of $K$. Viewing $\pow{N}{d}$ as the affine space $\AA^{\dim(\pow{N}{d})}$,  note that $h_\mathrm{Weil}$ is simply the usual Weil height on the weighted projective space $\Proj A\otimes_{R}K$. Although this is not the same as the Weil height $h$ on the usual projective completion $\PP^M$, one easily verifies that $h_\mathrm{Weil}\leq h\leq d!h_\mathrm{Weil}$.
From this we conclude the usual Northcott property.
\begin{lemma}\label{lem:northcott}
For any $B_1, B_2\geq 1$ there exist only finitely many $f\in\pow{N}{d}(\overline{K})$ such that $h_\mathrm{Weil}(f)\leq B_1$, and such that the coefficients of $f$ lie in an extension $L/K$ with $[L:K]\leq B_2$.
\end{lemma}

Now, for any $D\in \Div^*_K(\PP^N)$, define
\begin{equation}\label{eq:divheight}h(D)=\sum_{v\in M_K}\frac{[K_v:\QQ_v]}{[K:\QQ]}\lambda_{v}(D)\end{equation}
and
\begin{equation}\label{eq:canheight}\hat{h}_f(D)=\sum_{v\in M_K}\frac{[K_v:\QQ_v]}{[K:\QQ]}G_{f, v}(D).\end{equation}
We note that both of these are independent of $K$, and that both sums are well-defined, since $\lambda_v(D)=G_{f, v}(D)=0$ for all but finitely many $v\in M_K$.

\begin{theorem}\label{th:canheightprops}
For any $D\in \Div_K^*(\PP^N)$, and any $f\in \pow{N}{d}(K)$, we have
\begin{equation}\label{eq:canheightlimit}\hat{h}_f(D)=\lim_{n\to\infty}d^{-n}h(f^n(D)),\end{equation}
\begin{equation}\label{eq:canheighttrans}\hat{h}_f(f_*(D))=d\hat{h}_f(D),\end{equation}
\begin{equation}\label{eq:canheightlinear}\hat{h}_f\left(\sum e_i D_i\right)\leq \sum\hat{h}_f(D_i),\end{equation}
\begin{equation}\label{eq:canheightdiff}\hat{h}_f(D)=h(D)+O(h_\mathrm{Weil}(f)),\end{equation}
where the implied constant depends just on $N$ and $d$,
and
\[\hat{h}_f(D)=0\text{ if $D$ is preperiodic for $f$}.\]
\end{theorem}

\begin{proof}
For each place $v$, we have
\[G_{f, v}(D)=\lim_{n\to\infty} d^{-n}\lambda_v(f_*^n(D)).\]
For all but finitely many places, we have $\ph_v(f)=\lambda_v(D)=0$, and hence $\lambda_v(f_*^n(D))=0$ for all $n$. If we sum both sides over all places, suitably weighted, then each sum is actually finite, and so we may interchange the sum and the limit. This gives \eqref{eq:canheightlimit}, while \eqref{eq:canheighttrans} follows immediately from this.

Similarly,  Lemma~\ref{lem:nonarchgreensprops} and Lemma~\ref{lem:complexgreens} give
\[G_{f, v}(D)=\lambda_v(D)+O(\ph_v(f)),\]
where the implied constant depends only on $N$ and $d$. Summing this over all places gives \eqref{eq:canheightdiff}.

The property~\eqref{eq:canheightlinear} also follows directly from Lemma~\ref{lem:nonarchgreensprops} and Lemma~\ref{lem:complexgreens}.
Specifically, we have
\[G_{f, v}\left(\sum e_i D_i)\right)\leq\max\left\{G_{f, v}(D_i)\right\}\]
for each $v\in M_K$, and so
\begin{eqnarray*}
\hat{h}_f\left(\sum e_iD_i\right)&=&\sum_{v\in M_K}\frac{[K_v:\QQ_v]}{[K:\QQ]}G_{f, v}\left(\sum e_i D_i\right)\\
&=&\sum_{v\in M_K} \frac{[K_v:\QQ_v]}{[K:\QQ]} \max \left\{G_{f, v}(D_i)\right\}\\
&\leq& \sum_{v\in M_K}\frac{[K_v:\QQ_v]}{[K:\QQ]}\sum G_{f, v}(D_i)\\
&= & \sum \hat{h}_f(D_i).
\end{eqnarray*}
Note that the upper bound is not $\max\hat{h}_f(D_i)$, since the maximum value of the Greens functions may occur for different divisors at different places. But \eqref{eq:canheightlinear}  implies that $\hat{h}_f$ is bounded on the entire orbit of $D$ if $D$ is preperiodic, and \eqref{eq:canheighttrans} then implies that $\hat{h}_f(D)=0$.
\end{proof}

Note that it is not obvious, from what we have shown, that $\hat{h}_f(D)=0$ \emph{only if} $D$ is preperiodic for $f$, something which is true in the one-dimensional case. What we can conclude from the above is that $D$ is preperiodic for $f$ if and only if $\hat{h}_f(D)=0$ \emph{and} there is a bound on the degree of the irreducible components of $f_*^n(D)$ which is independent of $n$. We do not know whether or not there are examples of divisors $D$ with infinite forward orbit but $\hat{h}_f(D)=0$, and resolving this question would be of great interest.

Finally, we define
\begin{equation}\label{eq:critdef}h_\mathrm{crit}(f)=\hat{h}_f(C_f)=\sum_{v\in M_K} \frac{[K_v:\QQ_v]}{[K:\QQ]}\lambda_{\mathrm{crit}, v}(f),\end{equation}
noting again that $h_\mathrm{crit}(f)=0$ if $f$ is PCF. We speculate that this is the only case in which we have $h_\mathrm{crit}(f)=0$.
\begin{conjecture}\label{conj:critheight}
The PCF locus in $\pow{N}{d}(\CC)$ is precisely defined by the equality $h_\mathrm{crit}(f)=0$.
\end{conjecture}

Although Conjecture~\ref{conj:critheight} is a natural one, it is not necessary for the results in this paper.
\begin{proof}[Proof of Theorem~\ref{th:heights} and Corollary~\ref{cor:finiteness}]
Since we have
\[\lambda_{\mathrm{crit}, v}(f)=\ph_v(f)+O_{d, N, v}(1)\]
for all $v\in M_K$, and since the error terms vanish unless $v$ is non-archimedean or $p$-adic  for $p\leq d$, and since the errors depend only on the place of $\QQ$ which $v$ extends, we see that
\[h_\mathrm{crit}(f)=h_\mathrm{Weil}(f)+O_{d, N}(1)\]
as an immediate consequence of \eqref{eq:weildef} and \eqref{eq:critdef}. The finiteness claim of Corollary~\ref{cor:finiteness} is now a consequence of Lemma~\ref{lem:northcott}.
\end{proof}
We note that the finiteness in Lemma~\ref{lem:northcott}, coming from the Northcott property on $\PP^{\dim(\pow{N}{d})}$, is easily made effective (indeed the standard proof gives an algorithm, albeit a somewhat inefficient one, for computing sets of bounded height and algebraic degree). The effectiveness of Corollary~\ref{cor:finiteness}, then, comes down to two problems: making effective the error term in Theorem~\ref{th:heights} and finding an algorithm to test whether or not a given $f\in\pow{N}{d}(\overline{\QQ})$ is PCF. With regards to the first problem, we note that for a given $d$ and $N$, one can in principle perform the elimination of variables implied by Lemma~\ref{lem:nullstellensatz}, which is the only part of the proof of Theorem~\ref{th:heights} involving inexplicit constants. More generally, one can use the effective Nullstellensatz of Masser and W\"{u}stholtz~\cite{mw} to obtain a bound for these constant in terms of unknown $d$ and $N$, but since this bound is likely to be far from the truth, we have not gone so far as to write it down.
As to the problem of determining whether or not a given $f\in\pow{N}{d}(\overline{\QQ})$ is PCF, we note that this is intimately connected with Conjecture~\ref{conj:critheight}, and presents a tempting line of further inquiry.

%%%%%%%%%%%%%%%%%%%%%%%%%%%%%%
%%%%%%%%%%%%%%%%%%%%%%%%%%%%%%
%%%%%%%%%%%%%%%%%%%%%%%%%%%%%%
%%%%%%%%%%%%%%%%%%%%%%%%%%%%%%
%%%%%%%%%%%%%%%%%%%%%%%%%%%%%%
%%%%%%%%%%%%%%%%%%%%%%%%%%%%%%
%%%%%%%%%%%%%%%%%%%%%%%%%%%%%%
%%%%%%%%%%%%%%%%%%%%%%%%%%%%%%

\section{Computations}\label{sec:comp}

In this section we will outline the computations which lead to Theorem~\ref{th:computations}. Our computations were greatly aided by Manfred Minimair's package for computing Macaulay resultants in Maple \cite{mr}. The examples are presented in Table~\ref{table:examples}, along with a description of their critical orbits. These descriptions confirm both that the maps are PCF, and that they are pairwise non-conjugate, since conjugate morphisms have identical critical behaviour. The rest of this section outlines the confirmation that these are the only conjugacy classes. 

\begin{table}
\begin{center}
\begin{tabular}{@{}lll@{}} \toprule
$(a, b, c, d)$ & \multicolumn{2}{c}{Critical behaviour}  \\\toprule
\multirow{2}{*}{$(0, 0, 0, 0)$} & $D_1:\{x=0\}$ & $D_1\to D_1$\\
& $D_2:\{y=0\}$ & $D_2\to D_2$\\ \midrule
\multirow{4}{*}{$(0, 0, 0, -2)$} & $D_1:\{x=0\}$ & \multirow{2}{*}{$D_1\to D_1$} \\
& $D_2:\{y-1=0\}$ & \\
& $D_3:\{y+1=0\}$ & \multirow{2}{*}{$D_2\to D_3\to D_4\to D_4$} \\
& $D_4:\{y-3=0\}$ &\\\midrule
\multirow{6}{*}{$(-2, 0, 0, -2)$} & $D_1:\{x-1=0\}$ &\\
& $D_2:\{x+1=0\}$ & $D_1\to D_2\to D_3\to D_3$ \\
& $D_3:\{x-3=0\}$ &\\
& $D_4:\{y-1=0\}$ &\\
& $D_5:\{y+1=0\}$ & $D_4\to D_5\to D_6\to D_6$\\
& $D_6:\{y-3=0\}$ & \\\midrule
\multirow{3}{*}{$(0, 0, -1, 0)$} & $D_1:\{x=0\}$& \\
& $D_2:\{y=0\}$ & $D_1\to D_2\to D_3\to D_2$\\
& $D_3:\{y^2-x=0\}$ & \\\midrule
\multirow{3}{*}{$(0, 0, -2, 0)$} & $D_1:\{x=0\}$ & \\
& $D_2:\{y=0\}$ & $D_1\to D_2\to D_3\to D_3$\\
& $D_3:\{y^2-4x=0\}$ & \\\midrule
\multirow{2}{*}{$(0, -2, -2, 0)$} & $D_1:\{xy-1=0\}$ & \multirow{2}{*}{$D_1\to D_2\to D_2$} \\
& $D_2:\{x^2y^2-4x^3-4y^3+18xy-27=0\}$ &  \\\bottomrule
\end{tabular}
\end{center}\caption{Six PCF examples\label{table:examples}}
\end{table}

Let \[f_{a, b, c, d}(x, y)=\left(x^2+ax+by, y^2+cx+dy\right),\] where $a, b, c, d\in \QQ$. Then $C_f$ is defined by
\[xy+\frac{1}{2}xd+\frac{1}{2}ay+\frac{1}{4}(ad-bc)=0,\]
while  $f_*(C_f)$ is defined by
\begin{multline*}
x^2y^2-c^2x^3+\left(ac+\frac{1}{2}d^2\right)x^2y+\left(\frac{1}{2}a^2+bd\right)xy^2-b^2y^3\\
+\cdots +\frac{1}{256}(a^2d^2-27b^2c^2+4ca^3+4bd^3+18abcd)(ad-cb)^2=0,
\end{multline*}
where the intermediate terms have been suppressed. It follows from the results in Section~\ref{sec:local} (or just by an examination of the coefficients of $f_*(C_f)$  shown here) that for $f_{a, b, c, d}$ to be PCF, one must have $a, b, c, d\in\ZZ_p$ for any odd prime $p$. Similarly, one has $G_{f, 2}(C_f)>0$ unless all of the coefficients above turn out to be 2-adically integral. It follows that $b, c\in\ZZ_2$, while $a, d\in 2\ZZ_2$. So we have seen that $a, b, c, d\in\ZZ$, if $f_{a, b, c, d}$ is to be PCF, with $a$ and $d$ even.

Considering the usual (complex) absolute value, one sees easily that
\begin{multline*}
\max\left\{2\log|c|, 2\log|b|, \log|ac+\frac{1}{2}d^2|, \log|\frac{1}{2}a^2+bd|\right\}\\\geq 2\log\max\{|a|, |b|, |c|, |d|\} -2 \log\left(1+\sqrt{3}\right),
\end{multline*}
and so by the results of Section~\ref{sec:complex} we have
\[\lambda_\infty(f_*(C_f))\geq 2\ph_\infty(f) - 2\log\left(1+\sqrt{3}\right)-2\log 2-1,\]
although a more careful examination of the proof of Lemma~\ref{lem:NMineq} for polynomials of degree 4 shows that the lower bound can be increased by 1.
We then have either
\[\lambda_\infty(f_*(C_f))\leq \ph_v(f)+2\log 2,\]
resulting in $\ph_\infty(f)\leq 4\log 2+2\log(1+\sqrt{3})$, or else
\[0= G_{f, \infty}(f_*(C_f))\geq \lambda_\infty(f_*(C_f))-\log\left(\frac{1}{1-2^{-1/2}}\right),\]
from which $\ph_\infty(f)\leq \frac{3}{2}\log 2+\log(1+\sqrt{3})-\frac{1}{2}\log(2-\sqrt{2})$. In either case we get $\max\{|a|, |b|, |c|, |d|\}\leq 119$.

Since $a$ and $d$ are even, this leaves 808,890,481 four-tuples of integers to check (although some of these represent the same conjugacy class of map). We wrote a script in Sage~\cite{sage} which considered these tuples, and tried to show that $\lambda_{\mathrm{crit}, 2}(f_{a, b, c, d})$ was positive (by checking for non-integrality in the coefficients of $C_f$, $f_*(C_f)$, and some of the coefficients of $f^2_*(C_f)$) or that $\lambda_{\mathrm{crit}, \infty}(f_{a, b, c, d})$ was positive (by looking at the sizes of coefficients in the complex absolute value). Overnight, Sage pared our list down to only 127 values of $(a, b, c, d)\in\ZZ^4$ for which $f_{a, b, c, d}$ could possibly be PCF. 

The remaining tuples were treated individually using Minimair's MR package for Maple. In particular, for 83 of these tuples, an explicit calculation of $f_*^3(C_f)$ showed that the form defining this divisor (or some factor of it) had coefficients large enough to ensure that $\lambda_{\mathrm{crit}, \infty}(f)>0$, using the estimates from Section~\ref{sec:complex}. For another 19 tuples, it was sufficient to consider $f^4_*(C_f)$ and its factors, while in 3 instances it was necessary to compute at least some component of $f^5_*(C_f)$.  The remaining 22 tuples gave rise to PCF morphisms, and every one of these is conjugate to one of the examples presented in Theorem~\ref{th:computations}.

We note, as a remark on the examples of the form $(a, b, c, d)=(0, 0, *, 0)$, that if $f(x, y)=(x^2, y^2+cx)$, $D_w:\{y^2-w^2x=0\}$, and $E:\{x=0\}$, it is easy to show that $f_*(E)=2E$ and $f_*(D_w)=2D_{w^2+c}$. Since $C_f=2E+D_0$, it follows that $f$ is PCF if and only if the univariate map $w\mapsto w^2+c$ is, and the structures of the critical orbits coincide.

%\begin{multline*}
%
% Here is the full defining equation for the push-forward of the critical divisor:
%
%256x^2y^2-256c^2x^3+(256ac+128d^2)x^2y+(128a^2+256bd)xy^2-256b^2y^3\\
%+(-128c^2a^2+64acd^2+16d^4-384bc^2d)x^2\\+(128ca^3+64a^2d^2+256acbd+128bd^3+288b^2c^2)xy\\+(16a^4+64bda^2-384b^2ca-128b^2d^2)y^2\\
%+(-16c^2a^4+32d^2ca^3+(-128bc^2d+8d^4)a^2+(64d^3cb+144c^3b^2)a-120b^2c^2d^2+16d^5b)x\\+(16ca^5+8d^2a^4+64cbda^3+(32bd^3-120b^2c^2)a^2-128b^2cd^2a+144b^3c^2d-16d^4b^2)y\\+(4a^5cd^2+(d^4-8bc^2d)a^4+(16d^3cb+4c^3b^2)a^3+(4d^5b-62b^2c^2d^2)a^2\\+(-8d^4b^2c+72db^3c^3)a-27b^4c^4+4b^3c^2d^3)=0.
%\end{multline*}

\end{document}